\documentclass[a4paper,11pt,oneside]{article}
\usepackage{graphicx} 
\usepackage[a4paper]{geometry}
\usepackage[T1]{fontenc}
\usepackage[utf8]{inputenc}
\usepackage[english]{babel}
\usepackage{microtype}
\usepackage{quoting}
\usepackage{amsthm, amsmath, amssymb, amscd, mathrsfs, mathtools} 
\usepackage{dsfont}
\usepackage{booktabs, caption, graphicx} 
\usepackage{enumitem}
\usepackage{esint} 
\usepackage{framed}
\usepackage{braket}
\usepackage{tikz-cd}
\usepackage{centernot}
\usepackage{fancyhdr}
\usepackage{emptypage}
\usepackage{csquotes}
\usepackage{leftidx}
\usepackage{tensor}
\usepackage{stmaryrd}
\usepackage{setspace}

\definecolor{mygreen}{RGB}{13, 110, 53}
\definecolor{Green}{RGB}{0, 128, 0}
\definecolor{navyblue}{RGB}{0, 0, 128}
\definecolor{MyMulberry}{RGB}{197, 75, 140}

\usepackage[colorlinks,linkcolor=blue, urlcolor = navyblue, citecolor = Green]{hyperref}

\usepackage[backend=biber,style=alphabetic]{biblatex}
\usepackage{settings} 

\title{Weyl energy and connected sums of four-manifolds}

\author{Andrea Malchiodi and Francesco Malizia\thanks{Scuola Normale Superiore, Piazza dei Cavalieri 7, 56126 Pisa. e-mails: andrea.malchiodi@sns.it, francesco.malizia@sns.it}}

\date{}
\begin{document}

\maketitle

{\footnotesize
		\begin{abstract}	
			\noindent Given two closed, oriented Riemannian four-manifolds $(M,g_M)$ and $(Z,g_Z)$, which are not locally conformally flat and not both self-dual or both anti-self-dual, we prove that there exists a metric $g_Y$ on the connected sum $Y\cong M\#Z$ such that the Weyl energy of $g_Y$ is strictly smaller than the sum of Weyl energies of $g_M$ and $g_Z$. \\

			\vspace{3ex}
			

			
	\end{abstract}}

\begin{spacing}{0.75}
    {\tableofcontents}
\end{spacing}

\section{Introduction}\label{sec:intro}

One of the main goals of Riemannian geometry is to understand the relation between curvature and topology of smooth manifolds. In particular, when $M$ is a closed $4$-dimensional manifold, one of the best kinds of metrics one can hope to have are \emph{Einstein metrics}, that is, metrics $g$ satisfying
\begin{equation*}
    \mathrm{Ric}(g)=\lambda g,
\end{equation*}
for some real constant $\lambda$. In the Lorentzian context, Einstein metrics are solutions to the \emph{Einstein field equations} (with cosmological constant) of General Relativity; even in the Riemannian setting, they are intensively studied because of the homogeneity of their Ricci curvature, which provides good (orbifold) compactness properties resulting in a nice description of their moduli space, see e.g. \cite{anderson-1989-JAMS-einsteincompactness}, \cite{bando-kasue-nakajima-1989-Inventiones}. Roughly speaking, the above Einstein condition  corresponds, in a suitable gauge, to a \emph{determined} elliptic PDE system for the metric tensor (\cite[Chapter 5]{besse-book-2008}), whereas other constant curvature conditions correspond to either underdetermined (scalar curvature) or overdetermined (Riemann curvature tensor) constraints, resulting in abundance or rareness of their respective associated metrics. Moreover,  dimension $4$ appears to be the critical one for the Einstein condition because of its flexibility: indeed, in dimensions $2$ and $3$ it implies that the manifold has constant sectional curvature, while in dimension $5$ or higher there are no known obstructions to the existence of Einstein metrics and there might actually be \emph{too many} of them, see \cite{gromov-1999-GAFA-spaces-and-questions}.

Einstein metrics have a variational characterization: a metric is Einstein on a $4$-manifold $M$ if and only if it is a critical point for the \emph{normalized Einstein-Hilbert} functional
\begin{equation*}
    \mathcal{E}(g)=\mathrm{Vol}(M,g)^{-\frac{1}{2}}\int_M R_g\,dV_g,
\end{equation*}
where $R_g$ denotes the scalar curvature of $g$. One could then expect to obtain existence of Einstein metrics through variational methods; however, this is extremely difficult in practice: critical points of $\mathcal{E}$ possess infinitely many stable and unstable directions (\cite{schoen-89-montecatini}).

Even worse, it is known that not all closed $4$-manifolds do admit an Einstein metric: for instance, the \emph{Hitchin-Thorpe inequality} \cite{thorpe-1969-JMathMech-HitchinThorpe, hitchin-1974-JDG-HitchinThorpe}, 
\begin{equation*}
    2\chi(M)\geq 3\lvert\tau(M)\rvert,
\end{equation*}
is a well-known \emph{necessary} condition for the existence of Einstein metrics on a given \emph{oriented} manifold; here $\chi$ is the Euler characteristic and $\tau$ is the signature of the intersection form, see e.g. \cite{besse-book-2008} for further information.

This leads us to ask whether it is possible to find other canonical metrics which are {\em less obstructed} than Einstein ones. Recall that, on closed $4$-manifolds, the $(0,4)$-curvature tensor admits the following $L^2$-orthogonal decomposition:
\begin{equation}\label{eq:curv-decomp-4}
    \mathrm{Riem}=W \, \oplus \, E\owedge g \,\oplus\,\frac{R_g}{12}g\owedge g;
\end{equation}
here $\owedge$ is the Kulkarni-Nomizu product (defined as in \eqref{eq:kolknproduct}), $E:=\mathrm{Ric}-\frac{R_g}{4}g$ denotes the traceless Ricci tensor and $W$ is the Weyl curvature tensor, which is the \emph{conformally covariant} part of the Riemann tensor. Some interesting metrics are those arising as critical points of quadratic curvature functionals (\cite[Chapter 4]{besse-book-2008}); in particular, in this paper we will focus on the \emph{Weyl functional}
\begin{equation}\label{eq:weyl-funct}
    \w^M(g):=\int_M \abs{W^g}^2\,dV_g.
\end{equation}
Critical points of \eqref{eq:weyl-funct} are called \emph{Bach-flat metrics}, and they satisfy the Euler-Lagrange equation
\begin{equation}\label{eq:bach-flat-equation}
    B_{ij}:=-4\big(\nabla^{\alpha}\nabla^{\beta} W_{\alpha i j \beta}+\frac{1}{2}R^{\alpha\beta} W_{\alpha i j \beta}\big)=0, 
\end{equation}
where $B=B_{ij}$ is the so-called \emph{Bach tensor}, see \cite{bach-1921-mathz}, \cite{kobayashi-1985-J-Math-Soc-second-weyl-variation}. One can directly check that all Einstein metrics are Bach-flat, and so are metrics conformal to an Einstein one. Other special examples of Bach-flat metrics are \emph{self-dual} and \emph{anti-self-dual} metrics (see Remark \ref{rem:main-thm-remark} (a) and (c) below), whose existence, however, is also obstructed on some manifolds (by the signature formula \eqref{eq:hirz-sign-form}). Notice also that there are very few known examples of Bach-flat metrics which are not of the two types mentioned above, see \cite{abbena-garbiero-salamon-2013-CRAS-bachflat-lie-groups} and the references therein.

Now, to the best of the authors' knowledge, there are no known obstructions to the existence of Bach-flat metrics on closed $4$-manifolds. It is therefore tempting to try looking for critical points, in particular minima, of \eqref{eq:weyl-funct}. However, this is by no means an easy task: like all other curvature functionals, the Weyl functional is invariant under the action of the diffeomorphism group of $M$; moreover, it is also invariant under conformal transformations of the metric, so that actually $\w(g)=\w([g])$.
This double invariance generates analytic difficulties, even locally, as the Bach operator $B$ is \emph{not elliptic} in general: one has to consider, for instance, combinations like harmonic coordinates and constant scalar curvature conformal metric in order to (locally) turn $B$ into an elliptic operator, see e.g. \cite{anderson-2005-Math.Annalen-orbifoldcompactness}. This has also repercussions for the parabolic Bach flow, for which short-time existence and uniqueness have been proved via a conformal adjustment that deals with the gauge invariances along the lines of DeTurck's trick, see \cite{bahuaud-helliwell-2011-CPDE-fourth-order-flow-existence}, \cite{bahuaud-helliwell-2015-BullLMS-uniqueness-fourth-order}, \cite{chen-lu-qing-2023-AGAG-conf-Bach-flow}. \\
In the global picture, the study of minimizing sequences for the Weyl functional, or, more realistically, for a suitable regularized/subcritical version of it, requires to deal with the possible occurrence of \emph{collapses with bounded curvature} (in the sense of Cheeger and Gromov \cite{cheeger-gromov-1986-JDG-collapsing1,cheeger-gromov-1990-JGD-collapsing2}) or formation of \emph{orbifold singularities}, see the discussion in \cite{anderson-2006-AsianJ-canonical-metrics}. In particular, similarly to Einstein metrics, we know that, under suitable bounds on the curvature and volume, sequences of Bach-flat metrics are, \emph{a priori}, only \emph{orbifold precompact}, see \cite{anderson-2005-Math.Annalen-orbifoldcompactness}, \cite{tian-viaclovsky-2005-Inventiones}, \cite{tian-viaclovsky-2005-advances}, \cite{tian-viaclovsky-2008-CMH}. As a consequence, since general minimizing sequences will at least develop all the possible degenerations of the related moduli space of minima, it is necessary to understand under which assumptions it is possible to exclude the formation of orbifold singularities.

\bigskip

An instructive example of how to deal with some degeneration phenomena comes from the analysis of \emph{Willmore surfaces}, that is, immersed surfaces $f:\Sigma^\rho\to\R^n$ which are critical for the \emph{Willmore energy}
\begin{equation*}
    \mathbb{W}(f):=\frac{1}{4}\int_{\Sigma^\rho}\abs{H}^2\,d\mu,
\end{equation*}
where $H\in\R^n$ is the mean curvature vector, $\Sigma^\rho$ a closed surface of genus $\rho\in\N_{\geq0}$ and $\mu$ the induced surface measure. In \cite{simon-1986-proceed-existence-willmore-min-1,simon-1993-CAG-existence-Willmore-min-2}, Simon employed the direct method in the calculus of variations to minimize the Willmore energy among immersed surfaces of fixed genus $\rho\geq1$. Although in the  case $\rho=1$ he was able to prove the existence of a smooth minimal torus, the occurrence of topology degenerations in the limit was not excluded in the higher genus case $\rho\geq 1$. In other words, there was the possibility for a genus $\rho$ minimizing sequence to converge (in Hausdorff distance and weak measure sense) to a closed smooth embedding of \emph{lower genus} $\rho_0<\rho$. In general, given $\rho>1$, Simon \cite[Theorem 3.1]{simon-1993-CAG-existence-Willmore-min-2} proved that there exists a partition $\rho=\rho_1+\dots+\rho_m$ (with $\rho_1,\dots,\rho_m\geq1$) such that, letting
\begin{equation*}
    \beta_\rho:=\inf\{\mathbb{W}(f)\mid f:\Sigma^{\rho}\to\R^n\},
\end{equation*}
then each $\beta_{\rho_i}$ is \emph{attained} and
one has $\beta_\rho-4\pi=\beta_{\rho_1}+\dots+\beta_{\rho_m}-4\pi m$.

The occurrence of $m\geq2$ was later ruled out by Bauer and Kuwert in \cite{bauer-kuwert-2003-IMRN} (see also \cite{kusner-1996-book-contrib}) through the following strategy: they proved that, letting $f_i:\Sigma^{\rho_i}\to\R^n$, $i=1,2$ be two immersed surfaces which are both different from a round sphere, then it is possible to find an immersed surface $f:\Sigma^\rho\to\R^n$ with the topology of the connected sum (i.e. $\Sigma^\rho\cong\Sigma^{\rho_1}\#\Sigma^{\rho_2}$) satisfying
\begin{equation}\label{eq:willmore-gluing-ineq}
    \mathbb{W}(f)<\mathbb{W}(f_1)+\mathbb{W}(f_2)-4\pi.
\end{equation}
With this result at hand, assuming e.g. $\beta_\rho=\beta_{\rho_1}+\beta_{\rho_2}-4\pi$ in Simon's result and letting, for $i=1,2$, $f_i:\Sigma^{\rho_i}\to\R^n$ be two immersions minimizing $\mathbb{W}$, then \eqref{eq:willmore-gluing-ineq} readily gives a contradiction. Therefore, for any genus $\rho\geq0$, the infimum of Willmore's energy among immersions of $\Sigma^\rho$ into $\R^n$ is attained by a surface of genus $\rho$ (in case $\rho=0$ the infimum is achieved by the round sphere).

\bigskip

In this paper, we start an analogous program of studying the behavior of Weyl energy under the connected sum operation by establishing a partial analogue of \eqref{eq:willmore-gluing-ineq}.
The result is as follows.
\begin{theorem}\label{thm:main}
    Let $(M,g_M)$, $(Z,g_Z)$ be two $4$-dimensional closed, connected, oriented manifolds, and let $Y:=Z\#M$ denote their connected sum. Assume that $g_M, g_Z$ are not locally conformally flat and that they are not both self-dual or both anti-self-dual. Then there exists a metric $g_Y$ on $Y$ such that
    \begin{equation}\label{eq:en-inequality}
        \w^Y(g_Y)<\w^M(g_M)+\w^Z(g_Z).
    \end{equation}
\end{theorem}

\bigskip

Before describing the strategy behind the proof of Theorem \ref{thm:main}, we comment on its assumptions.

\begin{remark}\label{rem:main-thm-remark}
   \begin{itemize}
        \item[(a)] On a closed oriented $4$-manifold, the Weyl tensor $W$ decomposes into its \emph{self-dual} (SD) part $W_+$ and \emph{anti-self-dual} (ASD) part $W_-$ under the action of the Hodge $*$ operator: $W=W_++W_-$, see Section \ref{sec:prelims}. As a consequence, $(M,g)$ is said to be \emph{self-dual} (respectively, \emph{anti-self-dual}) if $W\equiv W_+$ on $M$ (respectively, $W\equiv W_-$). A metric is \emph{locally conformally flat} (LCF) if $W\equiv 0$. Of course, LCF metrics are both SD and ASD; moreover, these three conditions are all invariant under conformal transformations.
        \item[(b)] As it will be clear from the proof, the key idea which allows us to decrease the Weyl energy on $Y$ is to take advantage of the \emph{interaction} between the Weyl energies of $g_M$ and $g_Z$ in the gluing region; however, the interaction is always of higher order when the assumptions of Theorem \ref{thm:main} are not satisfied. 
        \item[(c)] In the particular case where both metrics are self-dual (respectively, anti-self-dual) and the second cohomology groups associated to their SD deformation complexes (respectively, ASD) are trivial,
        \begin{equation*}
            H^2_{\mathrm{SD}}(M,g_M)=\{0\}, \qquad H^2_{\mathrm{SD}}(Z,g_Z)=\{0\},
        \end{equation*}
        then there exists a SD (respectively, ASD) metric on $M\#Z$ by virtue of a theorem due to Donaldson and Friedman \cite{donaldson-friedman-1989-nonlin-SDgluing}, which relies upon the deformation theory of compact complex spaces. See also the paper \cite{floer-1991-JDG-CPsum} by Floer, who established the existence of SD metrics for connected sums of $\mathbb{C}\mathbb{P}^2$'s via gluing techniques. Since SD (and ASD) metrics are \emph{global minimizers} for the Weyl functional\footnote{This is an immediate consequence of the Hirzebruch signature formula \eqref{eq:hirz-sign-form}, cf. \cite[Proposition 5.12]{viaclovsky-notes-ParkCity}; in particular, all these metrics are Bach-flat.} and the techniques employed by Donaldson-Friedman and Floer are of perturbative nature, it follows that the infimum of Weyl energy is \emph{additive} in such cases. In other words, there is no hope to extend the result of Theorem \ref{thm:main} to the general case in which both manifolds are SD or both ASD. \\
        We also point out that there is a conjecture, attributed to Singer, which states that any SD manifold $(M,g)$ of positive Yamabe class, $Y(M,[g])>0$, actually satisfies $H^2_{\mathrm{SD}}(M,g)=\{0\}$ (and same in ASD case); see \cite{gover-gursky-2024-CRELLE} for an overview on this problem and related results.
    \end{itemize}
\end{remark}

\bigskip

We now briefly sketch the main ideas behind the proof of Theorem \ref{thm:main}. The procedure is inspired by the one adopted by Bauer and Kuwert in \cite{bauer-kuwert-2003-IMRN} for the Willmore functional; broadly speaking, we perform a connected sum between $M$ and $Z$ along a very small neck in a way that decreases the energy. To begin, by the conformal invariance of Weyl functional, we can perform the following operations on $g_M$ and $g_Z$ without affecting the Weyl energy:
\begin{itemize}
    \item invert\footnote{While for the Willmore functional a spherical inversion is performed, here the inversion, also called \emph{conformal blow-up}, is to be intended in (conformal) normal coordinates at a given point, see Section \ref{sec:gluing-setup} for details.}  one manifold, say $M$, at a point $p$; denote by $(N = M \setminus \{p\},g_N)$ the inverted manifold.
    \item Shrink $(N,g_N)$ into a manifold $(N,g_a)$ via the scaling $g_a:=a^2g_N$, for a small parameter $a\ll 1$. 
    \item Enlarge $(Z,g_Z)$ into a manifold $(Z,g_{b})$ via the scaling $g_b=b^{-2}g_Z:=\lambda^{-2}a^{-2}g_Z$, for some $\lambda>0$.
\end{itemize}
We then consider two thick annuli of relative width $\gamma\ll1$ (here by relative width we mean the ratio between inner radius and outer radius), one on the asymptotically flat (AF) end on $N$ and the other centered at a point $q\in Z$, and we glue them together, creating in this way a new manifold $Y\cong Z\#M$ (here we are discarding the region of the AF end of $N$ in the exterior of the annulus, as well as the ball centered at $q\in Z$ in the interior of the other annulus).

At this point, we define a new metric $g_Y$ on $Y$ which coincides with $g_a$, $g_b$ outside the gluing regions and which suitably interpolates between the two in the middle annulus. As in \cite{bauer-kuwert-2003-IMRN}, it turns out that a nice way to interpolate between $g_M$ and $g_Z$ is by asking $g_Y$ to be \emph{biharmonic} (with respect to the Euclidean metric and in suitable local coordinates) inside the gluing region. This is heuristically motivated as follows.
After scaling, the metrics $g_a$ and $g_b$ are almost flat in the gluing regions, that is, we can regard them as first-order perturbations of the Euclidean metric $g_E$. Indeed, letting $t:=a^2$, one has
\begin{equation*}
    g_a=g_{E}+tF+O(t^2), \qquad g_b=g_E +t\lambda^2 H +O(t^2),
\end{equation*}
where $F$ and $H$ are \emph{transverse and traceless} (TT) tensors. If we now expand the Weyl energy at the Euclidean metric (that is, at $t=0$), the first nonzero term is quadratic and depends upon the linearization of the Bach tensor, which, in this setting, is equal to the bilaplace operator, see \eqref{eq:weyl-exp-bilap-fin} below (the additional terms in \eqref{eq:weyl-exp-bilap-fin} are due to the presence of a boundary).

One then computes the energy balance $\w^Y(g_Y)-\w^M(g_M)-\w^Z(g_Z)$ between the {\em glued metric} and the starting ones; by construction, this difference will be equal to the Weyl energy of $g_Y$ in the annulus minus the Weyl energies of $g_M$ and $g_Z$ in the glued and cut out regions of $M$ and $Z$, respectively.

Looking at the first nonzero term in the expansion, we notice that its sign depends upon an \emph{interaction term} between the Weyl energy of $M$ and that of $Z$ evaluated at their respective gluing points. To make a long story short, the energy balance is of the following type (cf. Proposition \ref{prop:energy-balance}):
\begin{equation}\label{eq:en-bal-naive}
        \w^Y(g_Y)-\w^M(g_M)-\w^Z(g_Z)=a^4\Big(C-\frac{\pi^2}{9} \lambda^2 W^{\overbar{M}}(p)\stell W^Z(q) +O(\lambda^2\gamma^2) \Big) + \mathrm{h.o.t.},
    \end{equation}
where $\overbar{M}$ is the manifold $M$ with its \emph{opposite} orientation (this switches the SD/ASD Weyl tensors, $W^{\overbar{M}}_{\pm}=W^M_{\mp}$), $C$ is a constant depending upon $W^{\overbar{M}}(p)$, $\gamma$ is the aforementioned relative width of the gluing annulus and the interaction term has the form (cf. \cite{gursky-viaclovsky-2016-Advances})
\begin{equation}\label{eq:weyl-star-weyl}
    W^{\overbar{M}}(p)\stell W^Z(q):=\sum_{i,j,k,l}W^Z_{kijl}(q)\big(W^{\overbar{M}}_{kijl}(p)+W^{\overbar{M}}_{lijk}(p)\big).
\end{equation}
After the gluing procedure, we are identifying $T_p \overbar{M}$ with $T_q Z$, and we are expressing tensor components with respect to a 
common orthonormal basis. Notice that the right-hand side in \eqref{eq:weyl-star-weyl} \underline{depends upon the identification of the two tangent spaces}. Moreover, it clearly depends upon the basepoints themselves. This freedom will be crucially exploited in Section \ref{sec:proofs of main thms} in order to give \eqref{eq:weyl-star-weyl} a \emph{positive sign}, which allows to make \eqref{eq:en-bal-naive} \emph{negative} (for a large enough $\lambda$) and thus prove Theorem \ref{thm:main}.   

\begin{remark}
 One actually has to be careful with the various parameters; in particular, in the end it will be necessary to choose $0<a\ll\gamma\ll\lambda^{-1}<1$, see the proof of Theorem \ref{thm:main}.
\end{remark}
\begin{remark}
   Instead of working with a glued manifold $Y\cong M\#Z$, it will be easier in the next sections to perform a simpler gluing and consider $X\cong Z\# \overbar{M}$, see Remarks \ref{rem:gluingrem-1} and \ref{rem:manif-orientation}. 
\end{remark}

\begin{remark}
    The interaction term \eqref{eq:weyl-star-weyl} plays a fundamental role also in the paper \cite{gursky-viaclovsky-2016-Advances} by Gursky and Viaclovsky, from which we borrowed our notation. There, the authors construct critical metrics on connected sums of specific Einstein $4$-manifolds by means of a gluing (in the sense of finite dimensional reduction) procedure; the critical metrics in question are called $B^t$\emph{-flat} by the authors, and are critical points for the functional
    \begin{equation*}
       \mathcal{B}_t(g):= \int_M\abs{W^g}^2\,dV_g+t\int_M R_g^2\,dV_g.
    \end{equation*}
    In particular, \eqref{eq:weyl-star-weyl} appears in \cite{gursky-viaclovsky-2016-Advances} as a part of the coefficient (the other part is due to the scalar curvature and $t$) of the \emph{Kuranishi map} corresponding to the infinitesimal kernel parameter given by the freedom of scaling one manifold while keeping the other fixed.
    Notice that, in the present paper, the interaction term \eqref{eq:weyl-star-weyl} plays a similar role, as the parameter $\lambda$ in \eqref{eq:en-bal-naive} describes indeed the \emph{relative scaling} between the two manifolds. However, while Gursky and Viaclovsky do find critical points for the $\mathcal{B}_t$ functional above, here we are interested in obtaining estimates for the Weyl energy alone. Indeed, 
    the non-vanishing of term in \eqref{eq:weyl-star-weyl} seems to represent an {\em obstruction} to finding Bach-flat metrics on connected sums by 
    gluing methods. 
\end{remark}

\bigskip

We now briefly comment on possible consequences of Theorem \ref{thm:main} to the study of the infimum of Weyl's functional and to the loss of topology or formation of orbifold singularities along minimizing sequences; we also mention some related works and future perspectives.

\subsubsection*{Infimum of the Weyl functional}

An interesting quantity to study is the infimum of Weyl functional among all Riemannian metrics on a given $4$-manifold $M$:
\begin{equation*}\label{eq:inf-weyl-funct}
    \w(M):=\inf_{g}\w^M(g)=\inf_{g}\int_M\abs{W^{g_M}}^2\,dV_{g_M}.
\end{equation*}
In \cite{kobayashi-1985-J-Math-Soc-second-weyl-variation}, Osamu Kobayashi proved that the  infimum of Weyl energy is \emph{subadditive} with respect to the connected sum operation:
\begin{equation}\label{eq:kob-weyl-subadd}
    \w(M\#Z)\leq\w(M)+\w(Z).
\end{equation}

An immediate consequence of Theorem \ref{thm:main} is the following:

\begin{corollary}\label{cor:inf-weyl-metrics}
    Let $M,Z$ be two closed, oriented, connected $4$-manifolds and let $Y:=Z\#M$ denote their connected sum. Assume that $\w(M)$, $\w(Z)$ are strictly positive and that they are achieved at the conformal classes of $g_M$ and $g_Z$ respectively. Assume further that $g_M, g_Z$ are not both SD or both ASD. Then
    \begin{equation*}
        \w(Y)<\w(M)+\w(Z).
    \end{equation*}
\end{corollary}

\medskip

\begin{remark}
    The assumption $\w(M)>0$ implies that $M$ does not carry any locally conformally flat metric. On the other hand, notice that there are examples of closed $4$-manifolds $M$ (e.g. $M=\Sp^2\times \mathbb{T}^2$) which do not carry any LCF metric while satisfying $\w(M)=0$, see \cite[377]{kobayashi-1985-J-Math-Soc-second-weyl-variation}; in particular, the infimum of Weyl energy is never achieved on these manifolds.
\end{remark}

Corollary \ref{cor:inf-weyl-metrics} requires \emph{a priori} the existence of minimizers on the manifolds $M$ and $Z$. Thus, it would be interesting to develop a variational theory for minimizing sequences of the Weyl energy.

\medskip

After the work of Kobayashi, many results were obtained about the infimum of Weyl functional on manifolds admitting particular Riemannian metrics. One of the principal ones is due to Gursky, who proved a lower bound for the self-dual Weyl energy of positive Yamabe metrics over manifolds whose intersection form possesses at least one positive eigenvalue. Indeed, under these assumptions,  the following inequality holds (see \cite[Theorem 1]{gursky-1998-Annals}): for any conformal class $[g]$ with $Y(M,[g])>0$ (positive Yamabe class), one has\footnote{Notice that we have a factor $4$ of difference with the formula in \cite{gursky-1998-Annals} because here we are considering the norm of $W_+$ seen as a $(0,4)$-tensor instead that as an operator $\widehat{W}_+$ on $2$-forms.}
\begin{equation}\label{eq:gursky-ineq}
    \w_+^M(g):=\int_M\abs{W_+}^2\,dV_g\geq\frac{16\pi^2}{3}\big(2\chi(M)+3\tau(M)\big),
\end{equation}
with equality if and only if $[g]$ contains a (positive) K\"{a}hler-Einstein metric. Recall also that the study of the self-dual Weyl energy $\w_+$ is equivalent to that of Weyl's energy \eqref{eq:weyl-funct} because of the Hirzebruch signature formula
\begin{equation}\label{eq:hirz-sign-form}
    48\pi^2 \tau(M)=\int_M \big( \abs{W_+}^2-\abs{W_-}^2\big)\,dV_g;
\end{equation}
indeed, $\w^M(g)=-48\pi^2\tau(M)+2\w^M_+(g)$. A variant of Gursky's result was later obtained by LeBrun in \cite{lebrun-2015-JGA}, where, for some particular manifolds, \eqref{eq:gursky-ineq} was shown to hold for other families of conformal classes which may also be of \emph{negative} Yamabe type.
See also \cite{lebrun-2023-PreApplMathQ.-Weyl}, where it is proved that the infimum of Weyl functional is {\em small} (in the sense that \eqref{eq:gursky-ineq} fails) on many closed $4$-manifolds which also possess metrics of positive Yamabe class; this means that Gursky's inequality does not provide a lower bound for the infimum of Weyl energy in these cases. Examples of such manifolds are $m(\Sp^2\times \Sp^2)$ (i.e. connected sums of $m$ copies of $\Sp^2\times \Sp^2$) for $m$ large enough, or $(m\mathbb{C}\mathbb{P}^2)\#(n\overbar{\mathbb{C}\mathbb{P}}^2)$ for $m$ large and $n/m$ close to $1$; here $\overbar{\mathbb{C}\mathbb{P}}^2$ denotes the complex projective plane endowed with the \emph{other} orientation.

\begin{remark}
Since the product metric $g_{\Sp^2\times \Sp^2}$ on $\Sp^2\times \Sp^2$ is K\"{a}hler-Einstein, by Gursky's theorem one has
\begin{equation*}
    \int_{\Sp^2\times \Sp^2}\abs{W_+}^2\,dV_{g_{\Sp^2\times \Sp^2}}=\frac{16\pi^2}{3}8=\frac{128}{3}\pi^2,
\end{equation*}
so that, by repeated applications of Theorem \ref{thm:main}, one would get that $\w_+(m(\Sp^2\times\Sp^2))$ can be upper bounded by a number slightly below $\frac{128}{3}\pi^2m$ (this will be clear from the proof); on the other hand, LeBrun's Theorem A in \cite{lebrun-2023-PreApplMathQ.-Weyl} implies the much stronger bound $\w_+(m(\Sp^2\times\Sp^2))<\frac{64\pi^2}{3}(1+m)$ for $m$ large enough.
\end{remark}

Finally, we mention the paper \cite{taubes-1992-JDG-selfdual} by Taubes, where it is proved that, given any closed oriented $4$-manifold $M$, then $M\#m\overbar{\mathbb{C}\mathbb{P}}^2$ carries an ASD metric for $m$ large enough. In order to show this, the first step consists in decreasing the ASD Weyl energy of a fixed metric $g$ on $M$ by connect summing \emph{a lot of} copies of $\overbar{\mathbb{C}\mathbb{P}}^2$. While it is possible to obtain a decrease of ASD Weyl energy by arguing as in the proof of Theorem \ref{thm:main}\footnote{By the signature formula \eqref{eq:hirz-sign-form} and under the assumptions of Theorem \ref{thm:main}, the $L^2$ norms of $W_+$ and $W_-$ are both decreased in the gluing region.}, it is less clear whether one could actually obtain a \emph{quantitative} relative decrease as in Taubes's paper, where the gluing procedure is actually performed in a different way (with no use of biharmonic interpolations).

\bigskip

\subsubsection*{Obstructions to loss of topology and orbifold singularities along minimizing sequences}

In the same spirit of \cite{bauer-kuwert-2003-IMRN}, we would like to employ the strategy of Theorem \ref{thm:main} to study the formation of orbifold singularities along minimizing sequences for the Weyl energy (or for sequences of energy-minimizing Bach-flat metrics). For what follows, we refer the reader to \cite{tian-viaclovsky-2005-advances} or \cite{anderson-2005-Math.Annalen-orbifoldcompactness,anderson-2006-AsianJ-canonical-metrics}
for all the definitions and precise statements of the orbifold precompactness results for critical metrics. 

Assume for instance to have the simplest possible degeneration, that is, a sequence of unit-volume Bach-flat energy minimizing metrics $(g_i)_i$ on a manifold $Y$ which is Gromov-Hausdorff converging to a Bach-flat manifold $(Z,g_Z)$ while developing a single \emph{bubble} in the limit at some point $p\in Z$; here the bubble is a Bach-flat \emph{asymptotically flat} manifold $(N,g_N)$\footnote{We remark that this kind of behavior is ruled out in the Einstein case, where orbifold singularities are the only possible degeneration due to Bishop-Gromov's inequality.} and $Y$ is diffeomorphic to $Z\#N$. By the results of \cite{streets-2010-transactions-removal-singul-Bachflat} (see also \cite{ache-viaclovsky-2012-GAFA-obstruction-flatALE}) and the rigidity statement of positive mass theorem, we know that $(N,g_N)$ is AF of order $2$; moreover, the metric $g_Z$ extends smoothly at the (a priori) singular point $p$.
Also, it holds 
\begin{equation*}
    \w(Y)=\inf_{g}\w^Y(g)=\w^Z(g_Z)+\w^N(g_N),
\end{equation*}
cf. Theorems 4.10 and 4.11 in \cite{anderson-2006-AsianJ-canonical-metrics}.
In this setting, if we are able to glue back together $Z$ and $N$ in a way that saves energy, then we get an immediate contradiction, as the resulting metric on $Y$ would have Weyl energy strictly below $\w(Y)$. Hence the sequence $(g_i)_i$ cannot lose topology at $p$ in the first place. 

In other words, we would like to apply Theorem \ref{thm:main} in order to obtain an \emph{obstruction to the loss of topology}, but here the manifold $(N,g_N)$ is not closed. However, we can deal with this issue by arguing exactly as in \cite[1319]{streets-2010-transactions-removal-singul-Bachflat}: indeed, the fact that $g_N$ is AF of order $2$ allow us to \emph{conformally compactify} $(N,g_N)$ into a closed smooth manifold $(M,g_M)$ satisfying $\w^M(g_M)=\w^N(g_N)$. At this point, we can apply Theorem \ref{thm:main} and obtain an obstruction on $(N,g_N)$  and $(Z,g_Z)$.

More generally, this procedure of conformal compactification can also be performed in the ALE case, but this time we are not free to move the basepoint of the gluing. Indeed, the compactification only allow us to choose inverted conformal normal coordinates, for which we have a good expansion of the metric. As a consequence, we only get an obstruction on the product of Weyl tensors \emph{at the orbifold points}. This creates difficulties as the Weyl tensor (and some of its covariant derivatives) may in principle vanish at (a local lifting of) such points, leading to higher order obstructions. We plan to tackle these problems in future works.

\begin{remark}
    In the Einstein orbifold case (where no conformal normal coordinates are available), a good gauge for expanding the ALE metric at infinity has been obtained by Biquard and Hein in \cite{biquard-hein-2023-JDG-renormalizedvolume}. Subsequently, it has been employed by Ozuch in \cite{ozuch-2024-CPAM-obstructions-integrability-deform} in order to show that not all Einstein $4$-orbifolds can be obtained as Gromov-Hausdorff limits of smooth Einstein $4$-manifolds; see also the previous works \cite{biquard-2013-Inventiones-desingularization-einstein-1}, \cite{ozuch-2022-G&T-noncollapsed-degeneration-1}, \cite{ozuch-2022-G&T-noncollapsed-degeneration-2}. However, the ALE singularity models in this case are \emph{Ricci-flat} and they are known to be of order $4$ by the work of Cheeger and Tian, \cite[Theorem 5.103]{cheeger-tian-1994-Inventiones-coneatinfinity} (see also \cite{kroncke-szabo-2024-CalcVar-coord-RFconifolds}). In particular, the obstructions obtained in \cite{ozuch-2024-CPAM-obstructions-integrability-deform} {\em represent} the lack of integrability of certain infinitesimal Einstein deformations\footnote{They are the sum of the first error terms in the expansion of the ALE metric at infinity and the orbifold metric near the singularity, all in suitable gauges.} along the ALE end; moreover, as said in \cite[Remark 1.5]{ozuch-2024-CPAM-obstructions-integrability-deform}, these obstructions can be rewritten as an expression involving an interaction term like \eqref{eq:weyl-star-weyl} plus another quantity called \emph{renormalized volume} (introduced, in this particular setting, by \cite{biquard-hein-2023-JDG-renormalizedvolume}). Since Einstein metrics are also absolute minimizers of the $L^2$-norm of the Riemann curvature tensor,
    \begin{equation*}
        g\mapsto\int_M\abs{\mathrm{Riem}_g}^2\,dV_g,
    \end{equation*}
    (this follows from the Chern-Gauss-Bonnet formula), it is tempting to ask whether it is possible to recover the same obstructions by Ozuch via a suitable gluing procedure as for the case of Weyl energy.
\end{remark}

\bigskip

We end this Introduction with a brief description of the contents of each Section. In Section \ref{sec:prelims} we fix the notation and recall some known facts and preliminary results; in particular, we derive a formula for the expansion of the Weyl functional in TT directions at the Euclidean metric. In Section \ref{sec:gluing-setup}, we precisely describe how to perform the connected sum and we further derive the equations for the interpolating metric in the gluing annulus. After that, in Section \ref{sec:biharm-interp} we explicitly solve the biharmonic interpolation system and in Section \ref{sec:energy-balance} we use this information to compute the energy balance between the starting metrics and the one on the connected sum. Finally, in Section \ref{sec:proofs of main thms} we study the sign of \eqref{eq:weyl-star-weyl} and prove Theorem \ref{thm:main}.

\medspace \medspace \medspace 

\begin{center}
	{\bf Acknowledgments} 
\end{center}

\noindent 
The authors are supported by the PRIN Project 2022AKNSE4 {\em Variational and Analytical aspects of Geometric PDE} and are members of GNAMPA, as part of INdAM. A.M. is also supported by the project {\em Geometric problems with loss of compactness} from Scuola Normale Superiore. 
The authors are grateful to  Matthew Gursky for some inspiring discussions and many helpful remarks.

\section{Preliminaries}\label{sec:prelims}

We begin by fixing the notation and recalling some preliminary results.

\medskip

\textbf{Notation:} throughout this paper, we will adopt the Einstein summation convention, \emph{but only for repeated indices appearing once as a subscript and once as a superscript}. Thus, for instance, $\sum_i a_i b^i=a_ib^i$, but $\sum_i a_i b_i\not=a_i b_i$. This is important as, in many computations, we will have indices appearing as subscripts and superscripts together with other indices only appearing as subscripts or superscripts, so we \emph{forewarn} the reader about this. Related to this, we will usually denote local coordinates of a vector $x\in \R^n$ as $x^i$, but sometimes we will use the notation $x_i$ to make clear the fact that we are \emph{not} implicitly summing on the components.

\medskip

The $(1,3)$ and $(0,4)$ versions of the curvature tensor are defined by
\begin{gather*}
    R(X,Y)Z:=\nabla_X\nabla_Y Z-\nabla_Y\nabla_X Z-\nabla_{[X,Y]}Z, \\
    R(X,Y,Z,W):=g\big(R(X,Y)Z,W\big).
\end{gather*}
In local coordinates, the Riemann tensor will be written as $R_{ijkl}:=R(\partial_i,\partial_j,\partial_k,\partial_l)$ and the Ricci tensor as $R_{ij}$, while the scalar curvature of a metric $g$ will be denoted by $R_g$. Thus
\begin{equation*}
    R_{ijkl}=R\indices{_{ijk}^{\alpha}}g_{\alpha l}, \qquad  R_{ij}=R_{kijl}g^{kl}=R_{iklj}g^{kl}=R\indices{_{kij}^k}, \qquad R_g=R_{ij}g^{ij},
\end{equation*}
that is, the Ricci curvature is obtained by contracting the second and third components of the $(0,4)$-tensor (or, equivalently, the first and fourth ones).

As mentioned in the Introduction, on a $4$-manifold, the $(0,4)$-curvature tensor admits the $L^2$-orthogonal decomposition \eqref{eq:curv-decomp-4}, see e.g. \cite{besse-book-2008}. In particular, $\owedge$ is the \emph{Kulkarni-Nomizu product}, which takes two symmetric $(0,2)$ tensors as input and gives back a $(0,4)$ tensor with the same symmetries of the curvature tensor; it is defined by
\begin{equation}\label{eq:kolknproduct}
    (a\owedge b)(X,Y,V,W)=\frac{1}{2}\Big[ a(X,W) b(Y,V) + a(Y,V)b(X,W)-a(X,V)b(Y,W)-a(Y,W)b(X,V)\Big].
\end{equation}
Notice that the $\frac{1}{2}$-factor in the definition is not a canonical choice.

The decomposition \eqref{eq:curv-decomp-4} implies that the Weyl tensor has all traces equal to zero.

\medskip

The curvature tensor can be regarded as a symmetric endomorphisms on $2$-forms called \emph{curvature operator}:
\begin{equation*}
    \widehat{R}\colon \Lambda^2(T^*M)\to \Lambda^2(T^*M),
\end{equation*}
cf. \cite[Lecture 5]{viaclovsky-notes-ParkCity} or \cite{besse-book-2008}.
Indeed, given a local frame $\{e_1,\dots, e_n\}$ with local coframe $\{e^1,\dots, e^n\}$, then the family $\{e^i\wedge e^j\}_{i<j}$ defines a local frame for $\Lambda^2(T^*M)$ and we can define
\begin{equation*}
    \widehat{R}(e^i\wedge e^j):=\frac{1}{2}R_{ijlk}e^k\wedge e^l=\sum_{i<j}R_{ijlk}e^k\wedge e^l.
\end{equation*}
Notice the position reversal between $k$ and $l$. With this definition, $\widehat{R}$ coincides with the curvature operator on $2$-forms defined e.g. in \cite{viaclovsky-notes-ParkCity}: here there is no position reversal but the $(0,4)$-curvature tensor has the opposite sign with respect to our convention.

In the same way, we can also regard the Weyl tensor $W$ as a symmetric \emph{trace-free} endomorphism $\widehat{W}$ on $\Lambda^2(T^*M)$.
In particular, if $n=4$ and $M$ is oriented, then the Hodge star operator acting on 2-forms, $*\colon \Lambda^2(T^*M)\to \Lambda^2(T^*M)$, satisfies $*^2=\mathrm{Id}$ and, under its action, $\Lambda^2$ decomposes into
\begin{equation*}
    \Lambda^2(T^*M)=\Lambda^2_+(T^*M)\oplus \Lambda^2_-(T^*M),
\end{equation*}
which are the $\pm 1$-eigenspaces of $*$ respectively and satisfy $\dim(\Lambda^2_{\pm}(T^*M))=3$. The Weyl curvature operator decomposes accordingly as follows:
\begin{equation*}
    \widehat{W}=\widehat{W}_+ +\widehat{W}_-, \qquad \qquad \widehat{W}_{\pm}\colon \Lambda^2_{\pm}(T^*M)\to\Lambda^2_{\pm}(T^*M).
\end{equation*}
$\widehat{W}_+$, $\widehat{W}_-$ are called \emph{self-dual} and \emph{anti self-dual} Weyl curvatures respectively (as well as $W_+, W_-$ which are the associated $(0,4)$-tensors).

Finally, notice that the norm of a $(0,4)$ curvature-type tensor $\abs{T}^2=T^{ijkl}T_{ijkl}$ differs by a factor of $4$ from the norm $\big\lVert{\widehat{T}}\big\rVert^2$ of $T$ seen as a symmetric endomorphism of $\Lambda^2(T^*M)$:
\begin{equation*}
    \abs{T}^2=4\big\lVert{\widehat{T}}\big\rVert^2.
\end{equation*}

\subsection{Linearization of curvature tensors}

Here we recall the formulae for some linearized geometric quantities. The only one we really need is that of Weyl tensor, which however is obtained as a byproduct of the previous ones. These formulae are well-known in literature and we only recall them for the sake of completeness. Some references are, e.g., \cite{besse-book-2008} or \cite{catino-mastrolia-book}, but beware of the fact that they both employ a different convention for the curvature in local coordinates.

Throughout this subsection, we will assume that $(g_t)_t$ is a smooth family of Riemannian metrics on a manifold $M$ with $g_0=g$ and we let $h=\frac{d}{dt}\bigr\lvert_{t=0}g_t$ denote the linearization of $g_t$ at $t=0$; by construction, $h$ is a symmetric $(0,2)$-tensor.
We then have the following formulae:

    \textbullet \,\, Inverse of metric tensor:
    \begin{equation}\label{eq:inv-metr-lin}
        \frac{d}{d t}\biggr\rvert_{t=0}g_t^{ij}=-h^{ij}=-g^{\alpha i}g^{\beta j}h_{\alpha \beta}.
    \end{equation}
    
    \textbullet \,\, Christoffel symbols:
    \begin{equation}\label{eq:chr-symb-lin}
    \frac{d}{d t}\biggr\rvert_{t=0}  \leftidx{^{g_t}}\Gamma_{ij}^k =\frac{1}{2} g^{kl}\big(\nabla_i h_{lj}+\nabla_j h_{il}-\nabla_l h_{ij}\big).
    \end{equation}
    
    \textbullet \,\, $(1,3)$-curvature tensor:
    \begin{equation}\label{eq:curv-13-lin}
        \frac{d}{d t}\biggr\rvert_{t=0} \leftidx{^{g_t}}R\indices{_{ijk}^l}=\frac{1}{2} g^{sl}\big[\nabla^2_{ik}h_{js}+\nabla^2_{js}h_{ik}-\nabla^2_{is}h_{jk}-\nabla^2_{jk} h_{is}-R\indices{_{ijs}^r}h_{rk}-R\indices{_{ijk}^r}h_{sr}\big],
    \end{equation}
    where $\nabla^2_{ij}h_{kl}=(\nabla^2 h)_{ijkl}$. In order to prove this formula, we write down $R\indices{_{ijk}^l}$ in terms of Christoffel symbols and then we use \eqref{eq:chr-symb-lin} and the commutation formula
    \begin{equation*}
        \nabla^2_{ij}T_{kl}-\nabla^2_{ji} T_{kl}=-R\indices{_{ijk}^s}T_{sl}-R\indices{_{ijl}^s}T_{ks}.
    \end{equation*}
    
    \textbullet \,\, $(0,4)$-curvature tensor:
    \begin{equation}\label{eq:curv-04-lin}
        \frac{d}{d t}\biggr\rvert_{t=0} \leftidx{^{g_t}}R_{ijkl}=\frac{1}{2}\big[\nabla^2_{ik} h_{jl} +\nabla^2_{jl} h_{ik}-\nabla^2_{il} h_{jk} -\nabla^2_{jk} h_{il}+ R\indices{_{ijk}^s}h_{ls}- R\indices{_{ijl}^s}h_{sk}\big].
    \end{equation}
    To prove this, write $R_{ijkl}=R\indices{_{ijk}^s}g_{sl}$, linearize and use \eqref{eq:curv-13-lin}. Recall in particular that $\frac{d}{dt}$ \emph{does not commute} with index raising or lowering.
    
    \textbullet \,\, $(0,2)$-Ricci tensor:
    \begin{equation}\label{eq:ric-curv-02-lin}
        \frac{d}{d t}\biggr\rvert_{t=0} \leftidx{^{g_t}}R_{ij}=\frac{1}{2}\big[-\Delta h_{ij}-\nabla^2_{ij}(\tr h)+\nabla_i(\delta h)_j+\nabla_j(\delta h)_i + R\indices{_i^s} h_{sj}+R\indices{_j^s}h_{is}-2 R\indices{^r_{ij}^s}h_{rs}\big].
    \end{equation}
    Here $\delta$ stands for the divergence operator, which we define as $(\delta h)_i:=\nabla^k h_{ik}$; moreover, we are using the shorthands $R\indices{_i^s} h_{sj}=g^{st}R_{is} h_{tj}$ and $R\indices{^r_{ij}^s}h_{rs}=g^{rt} R\indices{_{rij}^s}h_{ts}$. To get \eqref{eq:ric-curv-02-lin}, utilize \eqref{eq:curv-13-lin} and the definition of $R_{ij}$; in particular, notice that $\frac{d}{dt}$ \emph{does not commute} with the trace operator.
    
    \textbullet \,\, Scalar curvature:
    \begin{equation}\label{eq:scal-curv-lin}
       \frac{d}{d t}\biggr\rvert_{t=0} R_{g_t}=-\Delta (\tr h)+ \delta^2(h)- h^{ij} R_{ij}.
    \end{equation}
    
    \textbullet \,\, Weyl tensor:
    \begin{align}
    \notag
        \frac{d}{d t}\biggr\rvert_{t=0} \leftidx{^{g_t}}W_{ijkl}&=\frac{1}{2}\big[\nabla^2_{ik}h_{jl} +\nabla^{2}_{jl} h_{ik}-\nabla^2_{il} h_{jk}- \nabla^2_{jk} h_{il}\big] \\
        \notag
        &+\frac{1}{2(n-2)}\Big\{ \big[-\Delta h_{ik}-\nabla^2_{ik}(\tr h)+\nabla_i(\delta h)_k +\nabla_k(\delta h)_i\big] g_{jl} \\
        \notag
        &+\big[-\Delta h_{jl}-\nabla^2_{jl}(\tr h)+\nabla_j(\delta h)_l+\nabla_l(\delta h)_j\big] g_{ik} \\
        \notag
        &-\big[-\Delta h_{il}-\nabla^2_{il}(\tr h)+\nabla_i(\delta h)_l+\nabla_l(\delta h)_i\big] g_{jk} \\
        \notag
        &-\big[-\Delta h_{jk}-\nabla^2_{jk}(\tr h)+\nabla_j(\delta h)_k+\nabla_k(\delta h)_j\big]g_{il}\Big\} \\
        \notag
        &+\frac{1}{(n-1)(n-2)}(-\Delta (\tr h) +\delta^2 h)(g_{jk}g_{il}-g_{ik}g_{jl}) \\
        \notag
        &+\frac{1}{2}\big[R\indices{_{ijk}^s} h_{ls}-R\indices{_{ijl}^s}h_{sk}\big]+\frac{1}{n-2}\big(R_{ik}h_{jl}+ R_{jl}h_{ik}-R_{il}h_{jk}-R_{jk} h_{il}\big) \\
        \notag
        &+\frac{1}{2(n-2)}\Big\{\big[R\indices{_i^r}h_{rk}+R\indices{_k^r}h_{ir}-2 R\indices{^r_{ik}^s}h_{rs}\big]g_{jl}+\big[R\indices{_j^r}h_{rl}+R\indices{_l^r}h_{jr}-2 R\indices{^r_{jl}^s}h_{rs}\big]g_{ik} \\
        \notag
        &-\big[R\indices{_i^r}h_{rl}+R\indices{_l^r}h_{ir}-2 R\indices{^r_{il}^s}h_{rs}\big]g_{jk}-\big[R\indices{_j^r}h_{rk}+R\indices{_k^r}h_{jr}-2 R\indices{^r_{jk}^s}h_{rs}\big]g_{il}\Big\} \\
        \label{eq:weyl-curv-04-lin}
        &-\frac{h^{ij}R_{ij}}{(n-1)(n-2)}\big(g_{jk}g_{il}-g_{ik}g_{jl}\big)+\frac{R_g}{(n-1)(n-2)}\big(h_{jk}g_{il}+g_{jk}h_{il}-h_{ik}g_{jl}-g_{ik}h_{jl}\big).
    \end{align}
In order to obtain this formula, we write down the Weyl tensor in local coordinates using \eqref{eq:curv-decomp-4},
\begin{equation}\label{eq:weyl-local-coord-expr}
    W_{ijkl}=R_{ijkl}-\frac{1}{n-2}\big(R_{jk} g_{il}+R_{il}g_{jk}-R_{ik}g_{jl}-R_{jl}g_{ik}\big)-\frac{R_g}{(n-1)(n-2)}(g_{jk}g_{il}-g_{ik}g_{jl}),
\end{equation}
and then we linearize, apply \eqref{eq:curv-04-lin}, \eqref{eq:ric-curv-02-lin} and \eqref{eq:scal-curv-lin} and regroup the terms as above.
\begin{remark}
    Notice that formula \eqref{eq:weyl-curv-04-lin} simplifies consistently when $h$ is TT (i.e. $\delta h=0$ and $\tr h=0$) or when the metric $g_0$ is flat. This will actually be the case in our later application.
\end{remark}

\subsection{Expansion of Weyl functional in TT directions at the Euclidean metric}

Let $\Omega\subset \R^4$ be an open, bounded, regular domain and consider $(g_t)_t$, a smooth family of metrics on $\Omega$, such that $g_0=g_E$ (Euclidean metric); let $h:=\frac{d}{dt}\big\rvert_{t=0}g_t$ be the linearization of $g_t$ at $t=0$. Assume moreover that $h$ is a transverse-traceless (TT) tensor.
Taylor-expanding $\w(g_t)=\int_\Omega \abs{W_{g_t}}^2\,dV_{g_t}$ at $t=0$ we get
\begin{equation}\label{eq:weyl-texp}
    \w(g_t)=\w(g_E)+t \left.\frac{d}{dt}\right\rvert_{t=0}\w(g_t)+\frac{t^2}{2}\frac{d^2}{dt^2}\biggr\rvert_{t=0}\w(g_t) +O(t^3).
\end{equation}
But now $\w(g_E)=0$ and (see e.g. the proof of Proposition 3.1 in \cite{kobayashi-1985-J-Math-Soc-second-weyl-variation})
\begin{equation}\label{eq:weyl-firstvar}
    \left.\frac{d}{dt}\right\rvert_{t=0}\w(g_t)=d\w(g_e)[h]=-4\int_\Omega \big( W_{\alpha i j \beta }\nabla^{\alpha}\nabla^{\beta} h^{ij}+\frac{1}{2}W_{\alpha i j \beta}R^{\alpha \beta}h^{ij}\big)\,dx=0,
\end{equation}
where the last equality is due to the fact that $W=W_{g_E}\equiv0$.
\begin{remark}
    On a \emph{closed manifold} $M$, assuming that $g_0$ is a critical point of $\w$, we can integrate \eqref{eq:weyl-firstvar} by parts and use the symmetry of $R^{\alpha\beta}$ to obtain
    \begin{equation*}
        -4\int_M \big(\nabla^{\alpha}\nabla^{\beta} W_{\alpha i j \beta}+\frac{1}{2}R^{\alpha\beta} W_{\alpha i j \beta}\big) h^{ij}\, dV_{g_0}=0.
    \end{equation*}
    In particular, we see that the Bach tensor \eqref{eq:bach-flat-equation} is the Euler-Lagrange operator for $\w$. Since in our case we are dealing with a domain $\Omega$ with boundary, we {\em stopped} before integrating by parts.
\end{remark}

In order to compute the second variation, we start from the first variation at a generic time $t$,
\begin{equation*}
    \frac{d}{dt}\w(g_t)=-4\int_\Omega \big( W_{\alpha i j \beta }\nabla^{\alpha}\nabla^{\beta} (\partial_t g)^{ij}+\frac{1}{2}W_{\alpha i j \beta}R^{\alpha \beta}(\partial_t g)^{ij}\big)\,dV_{g_t},
\end{equation*}
from which we get 
\begin{equation}\label{eq:weyl-secondvar-eucl}
    \frac{d^2}{dt^2}\biggr\rvert_{t=0}\w(g_t)=-4\int_\Omega \dot{W}_{\alpha i j \beta} \partial^\alpha \partial^\beta h^{ij}\,dx,
\end{equation}
where $\dot{W}:=\frac{d}{dt}\bigr\rvert_{t=0} W_{g_t}$ denotes the linearized Weyl tensor at $t=0$. Notice that here the flatness of the Euclidean metric was crucial in order to get rid of the other terms.
Using the formula \eqref{eq:weyl-curv-04-lin} for $\dot{W}$ with $g=g_E$ and the fact that $h$ is TT, we now have
\begin{equation}\label{eq:weyl-lin-TT-eucl}
    \dot{W}_{\alpha i j \beta}=\frac{1}{2}\big[\partial^2_{\alpha j}h_{i\beta}+\partial^2_{i \beta} h_{\alpha j}-\partial^2_{\alpha \beta} h_{ij}-\partial^2 _{ij} h_{\alpha \beta}\big]+\frac{1}{4}\big[\delta_{\alpha\beta} \Delta h_{ij}+\delta_{ij}\Delta h_{\alpha\beta}-\delta_{i\beta}\Delta h_{\alpha j}-\delta_{\alpha j}\Delta h_{i\beta}\big],
\end{equation}
thus, putting together \eqref{eq:weyl-texp}, \eqref{eq:weyl-firstvar}, \eqref{eq:weyl-secondvar-eucl} and \eqref{eq:weyl-lin-TT-eucl}, we obtain
\begin{align}
\notag
    \w(g_t)&=t^2\int_\Omega\big(\partial^2_{\alpha \beta} h_{ij}+\partial^2 _{ij} h_{\alpha \beta}-\partial^2_{\alpha j}h_{i\beta}-\partial^2_{i \beta} h_{\alpha j}\big)\partial^\alpha\partial^\beta h^{ij}\,dx \\
    \label{eq:weyl-exp-pre}
    &+\frac{t^2}{2}\int_\Omega\big(\delta_{i\beta}\Delta h_{\alpha j}+\delta_{\alpha j}\Delta h_{i\beta}-\delta_{\alpha\beta} \Delta h_{ij}-\delta_{ij}\Delta h_{\alpha\beta}\big)\partial^\alpha\partial^\beta h^{ij}\,dx +O(t^3).
\end{align}
Being $h$ a TT tensor, we easily see that
\begin{align}\label{eq:wvar-f1}
    \int_\Omega (\delta_{i\beta}\Delta h_{\alpha j}+\delta_{\alpha j}\Delta h_{i\beta}-\delta_{\alpha\beta} \Delta h_{ij}-\delta_{ij}\Delta h_{\alpha\beta}\big)\partial^\alpha\partial^\beta h^{ij}\,dx =-\int_\Omega \sum_{i,j}(\Delta h_{ij})^2\,dx.
\end{align}
For the other term of \eqref{eq:weyl-exp-pre}, after relabeling indices there holds
\begin{align}
\notag
    \int_\Omega \big(\partial^2_{\alpha \beta} h_{ij}+\partial^2 _{ij} h_{\alpha \beta}-&\partial^2_{\alpha j}h_{i\beta}-\partial^2_{i \beta} h_{\alpha j}\big)\partial^\alpha\partial^\beta h^{ij}\,dx \\
    \label{eq:wvar-f2}
    &=\int_\Omega \sum_{i,j} \abs{\nabla^2 h_{ij}}^2 \,dx +\int_\Omega \big( \partial^2_{ij} h_{\alpha\beta} -2\partial^2_{\alpha j}h_{i\beta}\big)\partial^\alpha\partial^\beta h^{ij}\,dx.
\end{align}
Moreover, using Bochner's identity on $h_{ij}$ and integrating by parts we get
\begin{align}\label{eq:wvar-f3}
    \int_{\Omega}\abs{\nabla^2 h_{ij}}^2\,dx=\int_\Omega (\Delta h_{ij})^2\,dx+\sum_\beta \int_{\partial \Omega}\partial^2_{\alpha\beta} h_{ij}\partial_\beta h_{ij}\nu^\alpha\,d\sigma -\int_{\partial \Omega}\Delta h_{ij} \partial_\alpha h_{ij} \nu^{\alpha}\, d\sigma,
\end{align}
(here $\nu=(\nu^1,\dots,\nu^4)$ is the exterior unit normal),
while an integration by parts and the fact that $h$ is divergence free also gives
\begin{align}\label{eq:wvar-f4}
    \int_\Omega \big( \partial^2_{ij} h_{\alpha\beta} -2\partial^2_{\alpha j}h_{i\beta}\big)\partial^\alpha\partial^\beta h^{ij}\,dx=\int_{\partial \Omega}\big(\partial^2_{ij} h_{\alpha \beta} -2\partial^2_{\beta j} h_{i \alpha}\big)\partial^\beta h^{ij}\nu^\alpha \,d\sigma.
\end{align}
Substituting \eqref{eq:wvar-f1}, \eqref{eq:wvar-f2}, \eqref{eq:wvar-f3} and \eqref{eq:wvar-f4} inside \eqref{eq:weyl-exp-pre} we obtain the following expansion for the Weyl energy:
\begin{align}
\notag
    \w(g_t)&=\frac{t^2}{2}\sum_{i,j}\int_{\Omega} (\Delta h_{ij})^2\,dx \\
    \label{eq:weyl-exp-biharm-fin}
    &+t^2\sum_{i,j,\beta}\int_{\partial \Omega}\Big[\big(\partial^2_{ij} h_{\alpha \beta}-2\partial^2_{\beta j} h_{i\alpha}+\partial^2_{\alpha\beta}h_{ij}\big)\partial_\beta h_{ij}-\partial^2_{\beta \beta} h_{ij} \partial_\alpha h_{ij}\Big]\nu^\alpha \,d\sigma+ O(t^3).
\end{align}
Finally, integrating by parts twice the biharmonic energy we get
\begin{align*}
    \int_{\Omega}(\Delta h_{ij})^2\,dx=\int_{\Omega}(\Delta^2 h_{ij}) h_{ij}\,dx -\int_{\partial \Omega} h_{ij} \partial_\alpha(\Delta h_{ij})\nu^\alpha\,d\sigma+\int_{\partial \Omega }\Delta h_{ij}\partial_\alpha h_{ij} \nu^\alpha \,d\sigma,
\end{align*}
so that \eqref{eq:weyl-exp-biharm-fin} rewrites as follows:
\begin{align}
\notag
    \w(g_t)&=\frac{t^2}{2}\sum_{i,j}\int_{\Omega}(\Delta^2 h_{ij}) h_{ij}\,dx-\frac{t^2}{2}\sum_{i,j,\beta}\int_{\partial \Omega}h_{ij} \partial^3_{\alpha\beta\beta} h_{ij}\nu^\alpha\,d\sigma \\
    \label{eq:weyl-exp-bilap-fin}
    &+t^2\sum_{i,j,\beta}\int_{\partial \Omega}\Big[\big(\partial^2_{ij} h_{\alpha \beta}-2\partial^2_{\beta j} h_{i\alpha}+\partial^2_{\alpha\beta}h_{ij}\big)\partial_\beta h_{ij}-\frac{1}{2}\partial^2_{\beta \beta} h_{ij} \partial_\alpha h_{ij}\Big]\nu^\alpha \,d\sigma+ O(t^3).
\end{align}

\section{The connected sum setup}\label{sec:gluing-setup}

In this section we are going to describe how to define a suitable {\em competitor metric} on the connected sum of two closed $4$-manifolds.
\begin{remark}\label{rem:gluingrem-1}
We notify the reader beforehand that, given two manifolds $M$ and $Z$, we will actually perform the connected sum in such a way that the resulting manifold $X$ will be diffeomorphic to $Z\#\overbar{M}$, where $\overbar{M}$ denotes the manifold $M$ endowed with its \emph{opposite orientation}, see \eqref{eq:identif-map} and Remark \ref{rem:manif-orientation}. As a consequence, the biharmonic interpolation and energy balance will be performed in this setting. This is not exactly what we want as we aim to decrease the Weyl functional on $Z\#M$; however, since $\w^M=\w^{\overbar{M}}$, in the end it will be sufficient to perform the same procedure with $\overbar{M}$ in place of $M$ in order to prove Theorem \ref{thm:main}. 
\end{remark}

To begin we recall that, given a closed manifold $(M,g)$ and a point $p\in M$, we got the following expansion of the metric $g$ in geodesic normal coordinates $\{z^i\}$ centered at $p$:
\begin{equation*}
    g_{ij}(z)=\delta_{ij}-\frac{1}{3} R_{kijl}(p)z^kz^l+O^{(3)}\big(\abs{z}^3\big), \qquad \text{as $\abs{z}\to0$}.
\end{equation*}

\begin{remark}
    Throughout this paper, we will use the symbol $O^{(\alpha)}\big(\abs{z}^\beta\big)$ to denote a function $f$ (or tensor sometimes) which satisfies $\abs{\nabla^k f}= O\big(\abs{z}^{\beta-k}\big)$ as $z$ approaches $0$ or infinity (depending on the situation) and for $k=0,1,\dots,\alpha$.
\end{remark}

For our purposes, it will be convenient to consider a refined expansion of the metric given by \emph{conformal normal coordinates}; indeed, by the results in \cite{lee-parker-1987-Yamabesurvey} (see the proof of their Theorem 5.1), there exists a conformal metric $\tilde{g}\in[g]$ such that, in geodesic normal coordinates for $\tilde{g}$ around $p$, one has 
\begin{equation}\label{eq:metr-exp-cnc}
    \tilde{g}_{ij}(z)=\delta_{ij}-\frac{1}{3}W^{\tilde{g}}_{kijl}(p)z^kz^l+O^{(3)}\big(\abs{z}^3\big), \qquad \text{as $\abs{z}\to0$}.
\end{equation}
Now, being the Weyl functional conformally invariant, we get $\w(g)=\w([g])$, which means that, in our procedure, we are free to start with any metric in the conformal class of $g$. As a consequence, without loss of generality, we can assume that the metric expansions in normal coordinates for our manifolds are always of the type \eqref{eq:metr-exp-cnc}.

\begin{remark}
    Looking at the proof of Theorem 5.1 in \cite{lee-parker-1987-Yamabesurvey}, we actually see that the conformal factor $f$ for the metric $\tilde{g}=e^f g$ satisfies $f(p)=0$; as a consequence, we also have $W^{\tilde{g}}_{kijl}(p)=e^fW^g_{kijl}(p)=W^g_{kijl}(p)$ by the conformal covariance of Weyl's tensor. 
\end{remark}

\bigskip

As already pointed out in the Introduction, in order to perform a connected sum which allows to reduce Weyl's energy, we want the manifolds to be almost flat in the gluing regions. To this purpose, following the approach adopted in \cite{bauer-kuwert-2003-IMRN} for the Willmore functional, we are going to invert one of the manifolds at a point and then glue its end to an enlarged copy of the other manifold. 

Let $(M,g_M)$ be a closed $4$-manifold. Consider a function $F:M\backslash\{p\}\to \R$ which is smooth, positive and such that $F(z)=\abs{z}^{-2}$ in (conformal) normal coordinates around $p$ for $\abs{z}$ small enough. Let now
\begin{equation*}
    (N,g_N):=\big(M\backslash\{p\}, F^2 g_M\big)
\end{equation*}
be the conformal blow-up of $(M,g_M)$ through $F$; then the pointwise identity $\abs{W^{g_N}}^2\, dV_{g_N} =\abs{W^{g_M}}^2\, dV_{g_M}$ on $M\backslash\{p\}$ implies that $\w^M(g_M)=\w^N(g_N)$. 
\begin{remark}
    If $(M,g)$ has positive scalar curvature, then (\cite[Theorem 2.8]{lee-parker-1987-Yamabesurvey}) there exists a unique  Green's function $G$ for the conformal Laplacian $L_g=-c_n\dg+R_g$, and one could consider the conformal blow-up of $M$ through $G$ as in \cite{gursky-viaclovsky-2016-Advances}. However, here we do not assume positive scalar curvature of $g_M$ and, moreover, the lower-order terms in the expansion of the Green function generate second order corrections to the inverted metric which we want to avoid.
\end{remark}

Let now $\{\bar{y}^i:=z^i/\abs{z}^2\}$ denote the \emph{inverted} (conformal) normal coordinates at $p$ and let $I\colon \R^4\backslash \{0\}\to\R^4\backslash \{0\}$ denote the inversion map, $I(z):=\frac{z}{\abs{z}^2}$. Then the following expansion holds for $g_N$:
\begin{equation}\label{eq:metr-inv-exp-cnc}
    (g_N)_{ij}(\bar{y})=\delta_{ij}-\frac{1}{3}W^{M}_{kijl}(p)\frac{\bar{y}^k\bar{y}^l}{\abs{\bar{y}}^4}+O^{(3)}\big(\abs{\bar{y}}^{-3}\big), \qquad \text{as $\abs{\bar{y}}\to\infty$};
\end{equation}
see the Appendix for a proof.

\bigskip

Let now $(Z,g_Z)$ be another closed oriented $4$-manifold and let $q\in Z$; as above, we can assume without loss of generality that $g_Z$ has the following expansion in normal coordinates $\{\bar{x}^i\}$ around $q$:
\begin{equation}\label{eq:metr-Z-exp}
     ({g}_Z)_{ij}(\bar{x})=\delta_{ij}-\frac{1}{3}W^{Z}_{kijl}(q)\bar{x}^k\bar{x}^l+O^{(3)}\big(\abs{\bar{x}}^3\big), \qquad \text{as $\abs{\bar{x}}\to0$}.
\end{equation}

We now want to enlarge $Z$ and shrink $N$ before performing the connected sum; to this purpose, consider two small parameters $0<a,b\ll 1$ (to be further specified later on) and the rescaled manifolds
\begin{equation*}
    (N,a^2g_N), \qquad (Z,b^{-2}g_Z).
\end{equation*}
If $a,b$ are small enough, we can assume that the normal coordinates $\{{x}^i:=b^{-1}\bar{x}^i\}$ of the scaled metric $b^{-2}g_Z$ at $q$ are defined in an open set containing the open ball $B_2(0)\subset \R^4$, and that the inverted normal coordinates $\{{y}^i:=a\bar{y}^i\}$ for $a^2g_N$ at $p$ are defined in $\R^4\backslash B_1(0)$. 
Then the following expansions hold:
\begin{equation}\label{eq:g_b-exp}
    (g_b)_{ij}(x)=\delta_{ij}-\frac{1}{3}b^2W^Z_{kijl}(q)x^kx^l+\eta^b_{ij}(x), \qquad \text{for $\abs{x}\leq 2$},
\end{equation}
where $\eta^b_{ij}$ satisfies
\begin{equation}\label{eq:err-g_b}
    \big\lvert\eta^b_{ij}(x)\big\rvert\abs{x}^{-3}+\big\lvert\nabla \eta^b_{ij}(x)\big\rvert \abs{x}^{-2}+\big\lvert\nabla^2 \eta^b_{ij}(x)\big\rvert\abs{x}^{-1}\leq C b^3, \quad \text{for $\abs{x}\leq 2$},
\end{equation}
and
\begin{equation}\label{eq:g_a-exp-ycoord}
    (g_a)_{ij}(y)=\delta_{ij}-\frac{1}{3}a^{-2}W^M_{kijl}(p)\frac{y^k y^l}{\abs{y}^4}+\tilde{\zeta}^a_{ij}(y), \qquad \text{for $\abs{y}\geq 1$},
\end{equation}
where $\tilde{\zeta}^a_{ij}$ satisfies
\begin{equation}\label{eq:err-g_a-ycoord}
    \big\lvert\tilde{\zeta}^a_{ij}(y)\big\rvert\abs{y}^{3}+\big\lvert\nabla \tilde{\zeta}^a_{ij}(y)\big\rvert \abs{y}^{4}+\big\lvert\nabla^2 \tilde{\zeta}^a_{ij}(y)\big\rvert\abs{y}^{5}\leq C a^{-3}, \quad \text{for $\abs{y}\geq 1$}.
\end{equation}

Let now $\gamma>0$ be another small parameter to be fixed later on and such that $0<a,b\ll\gamma< 1$. The gluing between $M$ and $Z$ is performed as follows: we remove from $(Z,g_b)$ the ball $\{x\mid \abs{x}<a\}$ centered at $q$ and from $(N,g_a)$ the region $\{y\mid\abs{y}> 2a^{-2}\}$; we then identify the annular region  $\{a\leq \abs{x}\leq 2\}\subset Z$ with the annular region $\{a^{-1}\leq \abs{y}\leq 2a^{-2}\}\subset N$ using the map $i$ defined below:
\begin{align}\label{eq:identif-map}
    i:\{a^{-1}\leq \abs{y}\leq& 2a^{-2}\}\to\{a\leq \abs{x}\leq 2\}, \qquad  x=i(y)=a^2y.
\end{align}
We denote by $X$ the new manifold obtained via this procedure.
\begin{remark}\label{rem:manif-orientation}
By construction, the manifold $X$ will be diffeomorphic to $Z\#\overbar{M}$, where $\overbar{M}$ denotes the (oriented) manifold $M$ endowed with its \emph{opposite orientation}. In particular, this is the same choice adopted in \cite{gursky-viaclovsky-2016-Advances}, see their Remark 9.1. The advantage over orientation-preserving identifications of the annular regions is due to the fact that here we have no permutations or sign reversals between the coordinates $\{y^i\}$ on $N$ and $\{x^i\}$ on $Z$.
\end{remark}
We also notice that, in the new coordinates $x=a^{2}y$, the expansion \eqref{eq:g_a-exp-ycoord} of $g_a$ becomes
\begin{equation}\label{eq:g_a-exp-fin}
    (g_a)_{ij}(x)=\delta_{ij}-\frac{1}{3}a^2W^M_{kijl}(p)\frac{x^kx^l}{\abs{x}^4}+\zeta^a_{ij}(x), \qquad \text{for $\abs{x}\geq a^2$},
\end{equation}
where now the error term $\zeta^a_{ij}(x):=\tilde{\zeta}^a_{ij}(a^{-2}x)$ satisfies
\begin{equation}\label{eq:err-g_a-fin}
    \abs{{\zeta}^a_{ij}(x)}\abs{x}^{3}+\abs{\nabla {\zeta}^a_{ij}(x)} \abs{x}^{4}+\abs{\nabla^2 {\zeta}^a_{ij}(x)}\abs{x}^{5}\leq C a^{3} \quad \text{for $\abs{x}\geq a^2$}.
\end{equation}
\bigskip

At this point, it only remains to define a suitable metric $g_X$ on $X$ which interpolates between $g_a$ and $g_b$ in the gluing region. To this purpose, for a given $a>0$, we can consider a smooth cutoff function $\chi_a\in C^{\infty}(\R)$ such that:
\begin{equation}\label{eq:cutoff-eta-def}
    \chi_a(t)=\begin{cases*}
        0 & \text{for $t\leq\frac{\sqrt{a}}{4}$} \\
        1 & \text{for $t\geq\frac{3}{4}\sqrt{a},$}
    \end{cases*}   \qquad \qquad \quad 
    \lvert\chi_a\rvert+\sqrt{a}\lvert\chi_a^{'}\rvert+a\lvert\chi_a^{''}\rvert\leq C,
\end{equation}
where $C>0$ does not depend upon $a$. We can now use $\chi_a$ to cut off the higher order terms in the expansions \eqref{eq:g_a-exp-fin}, \eqref{eq:g_b-exp} of $g_a$ and $g_b$ (with respect to the coordinates $\{x^i\}$) and define the following new metric $g_X$ in the outermost part of the gluing region (which we are now identifying with the region $\{a\leq \abs{x}\leq 2\}\subset Z$):
\begin{align}\label{eq:glued-metr-ext}
    (g_X)_{ij}(x):=\begin{cases*}
        \delta_{ij}-\frac{1}{3}a^2W^M_{kijl}(p)\frac{x^kx^l}{\abs{x}^4}+\zeta^a_{ij}(x)\chi_a(\gamma-\abs{x}) & \text{for $a\leq\abs{x}<\gamma$,} \\[0.5ex]
        \delta_{ij}-\frac{1}{3}b^2 W^Z_{kijl}(q)x^kx^l+\eta^b_{ij}(x)\chi_a(\abs{x}-1) & \text{for $1\leq\abs{x}\leq2$.}
    \end{cases*}
\end{align}
Notice that actually $g_X=g_a$ for $\abs{x}\leq\gamma-\sqrt{a}$ and $g_X=g_b$ for $\abs{x}\geq 1+\sqrt{a}$.

At this point, we look for a suitable interpolating metric in the annular region $\{\gamma<\abs{x}<1\}$: we define it to be the \emph{biharmonic interpolation} between the two metrics with $C^1$ matching at the boundary. In other words, we define it as $(g_X)_{ij}:=w_{ij}$, where (using the shorthands $W^M_{kijl}, W^Z_{kijl}$ for $W^M_{kijl}(p), W^Z_{kijl}(q)$) $w_{ij}$ is the unique solution to
\begin{equation}\label{eq:interp-system}
    \begin{cases*}
        \Delta^2 w_{ij}=0 & \text{in $\gamma<\abs{x}<1$} \\[0.5ex]
        w_{ij}=\delta_{ij}-\frac{1}{3}a^2W^M_{kijl}\frac{x^kx^l}{\abs{x}^4} & \text{in $\abs{x}=\gamma$} \\[0.5ex]
        w_{ij}=\delta_{ij}-\frac{1}{3}b^2W^Z_{kijl}x^kx^l & \text{in $\abs{x}=1$} \\[0.5ex]
        \partial_r w_{ij}=\frac{2}{3}a^2 W^M_{kijl}\frac{x^kx^l}{\abs{x}^5} & \text{in $\abs{x}=\gamma$} \\[0.5ex]
        \partial_r w_{ij}=-\frac{2}{3}b^2 W^Z_{kijl}\frac{x^k x^l}{\abs{x}} & \text{in $\abs{x}=1$}. \\
    \end{cases*}
\end{equation}
Here $r=\abs{x}$ and we notice that
\begin{equation*}
    \partial_r\Big(W^M_{kijl}\frac{x^kx^l}{\abs{x}^4}\Big)=-2W^M_{kijl}\frac{x^kx^l}{\abs{x}^5}, \qquad \partial_r\Big(W^Z_{kijl}x^kx^l\Big)=2 W^Z_{kijl}\frac{x^kx^l}{\abs{x}}.
\end{equation*}
In particular, we are requiring $w_{ij}$ to coincide at the boundary with the metric $(g_X)_{ij}$ (defined in \eqref{eq:glued-metr-ext}) up to the first order. The crucial thing is that we are actually able to \emph{explicitly compute} the solution $w=(w_{ij})_{i,j}$ of \eqref{eq:interp-system}, and this will allow us to compute the Weyl energy balance; we will solve \eqref{eq:interp-system} in Section \ref{sec:biharm-interp}.  

We can now complete the definition of the metric $g_X$ on $X$ by setting  
\begin{equation}\label{eq:g_X-def}
    g_X:=\begin{cases*}
        g_b & \text{in $Z\backslash\{x\mid \abs{x}<2\} $} \\
        g_a & \text{in $N\backslash\{y\mid \abs{y}>a^{-1}\}=N\backslash\{x\mid \abs{x}>a\} $} \\
        w   &  \text{in $\{\gamma<\abs{x}<1\}$} \\
        \text{as in \eqref{eq:glued-metr-ext}} & \text{in the remaining region}.
    \end{cases*}
\end{equation}
We will compare $\w^X(g_X)$ to $\w^N(g_a)+\w^Z(g_b)=\w^M(g_M)+\w^Z(g_Z)$ in Section \ref{sec:energy-balance}.
\begin{remark}
    Notice that $g_X$ is \emph{not} a smooth ($C^2$) metric: indeed, the second derivatives of $g_X$ may be discontinuous at the boundary of the annulus $\{\gamma<\abs{x}<1\}$. Nevertheless, the metric is globally $W^{2,\infty}$, so, if we are able to show that $\w^X(g_X)<\w^N(g_a)+\w^Z(g_b)$, then we can also find by density a smooth metric $\tilde{g}_X$ on $X$ satisfying the same inequality. 
\end{remark}

\section{Biharmonic interpolation}\label{sec:biharm-interp}

In this section we are going to explicitly compute the solution of system \eqref{eq:interp-system}. 

We begin by noticing that the boundary conditions of \eqref{eq:interp-system} can always be expressed as linear combinations of second spherical harmonics, which, by definition, are the restriction to $\Sp^3$ of homogeneous harmonic polynomials in $\R^4$ of degree $2$. Recall that there are $9$ linearly independent second spherical harmonics on $\Sp^3$:
\begin{align}
\notag
    &\phi_{k,l}(x):=\text{restriction to $\Sp^3$ of $x^kx^l$ ($6$ functions)}, \\
    \label{eq:sph-harm-def}
    &\psi_{k,l}(x):=\text{restriction to $\Sp^3$ of $(x^k)^2-(x^l)^2$ ($3$ functions)}.
\end{align}
Using the trace-free property and the symmetries of Weyl's tensor, we can rewrite system \eqref{eq:interp-system} as follows:

\medskip

\textbullet \,\, if $i=j$, calling $s,t,u$ the remaining three indices in $\{1,\dots,4\}$ (so that $\{i,s,t,u\}$ are all distinct), we have
\begin{equation}\label{eq:interp-system-iicase}
    \begin{cases*}
        \Delta^2 w_{ii}=0 & \text{in $\gamma<\abs{x}<1$} \\
        w_{ii}=1-\frac{1}{3}a^2\gamma^{-2}\Big(2W^M_{siit}\phi_{s,t}+2W^M_{siiu}\phi_{s,u}+2W^M_{tiiu}\phi_{t,u}+W^M_{siis}\psi_{s,u}+W^{M}_{tiit}\psi_{t,u}\Big) & \text{in $\abs{x}=\gamma$} \\
        w_{ii}=1-\frac{1}{3}b^2\Big(2W^Z_{siit}\phi_{s,t}+2W^Z_{siiu}\phi_{s,u}+2W^Z_{tiiu}\phi_{t,u}+W^Z
        _{siis}\psi_{s,u}+W^{Z}_{tiit}\psi_{t,u}\Big) & \text{in $\abs{x}=1$} \\
        \partial_r w_{ii}=\frac{2}{3}a^2\gamma^{-3}\Big(2W^M_{siit}\phi_{s,t}+2W^M_{siiu}\phi_{s,u}+2W^M_{tiiu}\phi_{t,u}+W^M_{siis}\psi_{s,u}+W^{M}_{tiit}\psi_{t,u}\Big) & \text{in $\abs{x}=\gamma$} \\
        \partial_r w_{ii}=-\frac{2}{3}b^2 \Big(2W^Z_{siit}\phi_{s,t}+2W^Z_{siiu}\phi_{s,u}+2W^Z_{tiiu}\phi_{t,u}+W^Z
        _{siis}\psi_{s,u}+W^{Z}_{tiit}\psi_{t,u}\Big) & \text{in $\abs{x}=1$}. 
    \end{cases*}
\end{equation}
This suggests to look for a solution of \eqref{eq:interp-system-iicase} of the type
\begin{align}\label{eq:sol-ii-form}
    w_{ii}=1+f_{s,t}(r)\phi_{s,t}+f_{s,u}(r)\phi_{s,u}+f_{t,u}(r)\phi_{t,u}+\tilde{f}_{s,u}(r)\psi_{s,u}+\tilde{f}_{t,u}(r)\psi_{t,u},
\end{align}
where the $f$'s and $\tilde{f}$'s are purely radial functions $(r=\abs{x})$.

\medskip

\textbullet \,\, if $i\not=j$, calling $s,t$ the remaining two indices in $\{1,\dots,4\}$, we have
\begin{equation}\label{eq:interp-system-ijcase}
    \begin{cases*}
        \Delta^2 w_{ij}=0 & \text{in $\gamma<\abs{x}<1$} \\
        \!\begin{aligned}
        w_{ij}=-\frac{1}{3}a^2\gamma^{-2}\Big(W^M_{jiji}\phi_{j,i}+W^M_{jijs}\phi_{j,s}+&W^M_{jijt}\phi_{j,t}+W^M_{siji}\phi_{s,i}+W^M_{tiji}\phi_{t,i} \\
        &+(W^M_{sijt}+W^M_{tijs})\phi_{s,t}+W^M_{sijs}\psi_{s,t}\Big)
        \end{aligned} & \text{in $\abs{x}=\gamma$} \\
        \!\begin{aligned}
        w_{ij}=-\frac{1}{3}b^2\Big(W^Z_{jiji}\phi_{j,i}+W^Z_{jijs}\phi_{j,s}+&W^Z_{jijt}\phi_{j,t}+W^Z_{siji}\phi_{s,i}+W^Z_{tiji}\phi_{t,i} \\
        &+(W^Z_{sijt}+W^Z_{tijs})\phi_{s,t}+W^Z_{sijs}\psi_{s,t}\Big)
        \end{aligned} & \text{in $\abs{x}=1$} \\
        \!\begin{aligned}
        \partial_rw_{ij}=\frac{2}{3}a^2\gamma^{-3}\Big(W^M_{jiji}\phi_{j,i}+W^M_{jijs}\phi_{j,s}+&W^M_{jijt}\phi_{j,t}+W^M_{siji}\phi_{s,i}+W^M_{tiji}\phi_{t,i} \\
        &+(W^M_{sijt}+W^M_{tijs})\phi_{s,t}+W^M_{sijs}\psi_{s,t}\Big)
        \end{aligned} & \text{in $\abs{x}=\gamma$} \\
        \!\begin{aligned}
        \partial_r w_{ij}=-\frac{2}{3}b^2\Big(W^Z_{jiji}\phi_{j,i}+W^Z_{jijs}\phi_{j,s}+&W^Z_{jijt}\phi_{j,t}+W^Z_{siji}\phi_{s,i}+W^Z_{tiji}\phi_{t,i} \\
        &+(W^Z_{sijt}+W^Z_{tijs})\phi_{s,t}+W^Z_{sijs}\psi_{s,t}\Big)
        \end{aligned} & \text{in $\abs{x}=1$}. 
    \end{cases*}
\end{equation}
This suggests to look for a solution of \eqref{eq:interp-system-ijcase} of the type
\begin{align}\label{eq:sol-ij-form}
    w_{ij}=f_{j,i}(r)\phi_{j,i}+f_{j,s}(r)\phi_{j,s}+f_{j,t}(r)\phi_{j,t}+f_{s,i}(r)\phi_{s,i}+f_{t,i}(r)\phi_{t,i}+f_{s,t}(r)\phi_{s,t}+\tilde{f}_{s,t}(r)\psi_{s,t},
\end{align}
where again the $f$'s and $\tilde{f}$'s are purely radial functions.

\bigskip

Consider now a general function $F$ on $B_1\backslash B_\gamma\subset \R^4$ of the form of \eqref{eq:sol-ii-form}, \eqref{eq:sol-ij-form}, namely
\begin{equation*}
    F(x)=\sum_\alpha f_{\alpha}(r)\phi_\alpha(\theta)+ C, 
\end{equation*}
where $C$ is a constant, $r=\abs{x}$, $\theta=\frac{x}{\abs{x}}\in\Sp^3$ and the $\phi_\alpha$'s are second spherical harmonics (any type). \\[0.5ex] 
Using the fact that each $\phi_\alpha$ satisfies $\Delta_{\Sp^3}\phi_\alpha=-8\phi_\alpha$, we see that
\begin{align*}
    \Delta^2 F=\sum_\alpha \phi_\alpha(\theta)\underbrace{\Big( f_{\alpha}^{(4)}(r)+\frac{6}{r}f_{\alpha}^{(3)}(r)-\frac{13}{r^2}f_{\alpha}^{''}(r)-\frac{19}{r^3}f_{\alpha}^{'}(r)+\frac{64}{r^4}f_\alpha(r)\Big),}_{=:T(f_\alpha)(r)}
\end{align*}
and it is now easy to show that all the solutions to the ODE $T(f_\alpha)=0$ are of type
\begin{align*}
    f_\alpha(r)=\frac{c_{1,\alpha}}{r^4}+\frac{c_{2,\alpha}}{r^2}+c_{3,\alpha} r^2 +c_{4,\alpha} r^4,
\end{align*}
where $c_{1,\alpha},c_{2,\alpha},c_{3,\alpha},c_{4,\alpha}\in\R$ are constants which depend upon the boundary data.
In particular, in the situation described above we have:

\bigskip

\textbullet \,\, if $i=j$, for $(\alpha,\beta)\in\{(s,t),(s,u),(t,u)\}$,
\begin{equation}\label{eq:bd-ii-phi}
    \begin{cases*}
        f_{\alpha,\beta}(\gamma)=-\frac{2}{3}a^2\gamma^{-2}W^M_{\alpha i i \beta} \\
        f_{\alpha,\beta}^{'}(\gamma)=\frac{4}{3}a^2\gamma^{-3}W^M_{\alpha i i \beta} \\
        f_{\alpha,\beta}(1)=-\frac{2}{3}b^2 W^Z_{\alpha i i \beta} \\
        f^{'}_{\alpha,\beta}(1)=-\frac{4}{3}b^2 W^Z_{\alpha i i \beta},
    \end{cases*}
\end{equation}
while, for $(\alpha,\beta)\in\{(s,u),(t,u)\}$, one has
\begin{equation}\label{eq:bd-ii-psi}
    \begin{cases*}
        \tilde{f}_{\alpha,\beta}(\gamma)=-\frac{1}{3}a^2\gamma^{-2}W^M_{\alpha i i \alpha} \\
        \tilde{f}_{\alpha,\beta}^{'}(\gamma)=\frac{2}{3}a^2\gamma^{-3}W^M_{\alpha i i \alpha} \\
        \tilde{f}_{\alpha,\beta}(1)=-\frac{1}{3}b^2 W^Z_{\alpha i i \alpha} \\
        \tilde{f}^{'}_{\alpha,\beta}(1)=-\frac{2}{3}b^2 W^Z_{\alpha i i \alpha}.
    \end{cases*}
\end{equation}

Writing now 
\begin{gather}
\label{eq:rad-funct-eq}
    f_{\alpha,\beta}(r)=\frac{c_{1,\alpha,\beta}}{r^4}+\frac{c_{2,\alpha,\beta}}{r^2}+ c_{3,\alpha,\beta} r^2 +c_{4,\alpha,\beta} r^4 \\
    \label{eq:raf-fun-tilde}
    \tilde{f}_{\alpha,\beta}(r)=\frac{\tilde{c}_{1,\alpha,\beta}}{r^4}+\frac{\tilde{c}_{2,\alpha,\beta}}{r^2}+ \tilde{c}_{3,\alpha,\beta} r^2 +\tilde{c}_{4,\alpha,\beta} r^4,
\end{gather}
from the general formula \eqref{eq:rad-funct-eq} of $f_{\alpha,\beta}$ and \eqref{eq:bd-ii-phi} we obtain the following system:
\begin{equation*}
    \begin{cases*}
        \frac{c_{1,\alpha,\beta}}{\gamma^4}+\frac{c_{2,\alpha,\beta}}{\gamma^2}+ c_{3,\alpha,\beta} \gamma^2 +c_{4,\alpha,\beta} \gamma^4 =-\frac{2}{3}a^2\gamma^{-2}W^M_{\alpha i i \beta} \\
        -\frac{4c_{1,\alpha,\beta}}{\gamma^5}-\frac{2c_{2,\alpha,\beta}}{\gamma^3}+ 2c_{3,\alpha,\beta} \gamma +4c_{4,\alpha,\beta} \gamma^3=\frac{4}{3}a^2\gamma^{-3}W^M_{\alpha i i \beta} \\
        c_{1,\alpha,\beta}+c_{2,\alpha,\beta}+ c_{3,\alpha,\beta} +c_{4,\alpha,\beta}=-\frac{2}{3}b^2W^Z_{\alpha i i \beta} \\
        -4c_{1,\alpha,\beta}-2c_{2,\alpha,\beta}+ 2c_{3,\alpha,\beta} +4c_{4,\alpha,\beta}=-\frac{4}{3}b^2W^Z_{\alpha i i \beta},
    \end{cases*}
\end{equation*}
which can be rewritten in the general matrix form
\begin{equation}\label{eq:lin-syst-eq}
    A_\gamma \mathbf{c}=a^2\mathbf{v},
\end{equation}
where 
\begin{equation}\label{eq:agamma-matrix}
    A_\gamma:=\begin{pmatrix}
        1 & \gamma^2 & \gamma^6 & \gamma^8 \\[0.5ex]
        -4 & -2\gamma^2 & 2\gamma^6 & 4\gamma^8 \\[0.5ex]
        1 &1 & 1 &1 \\[0.5ex]
        -4 & -2 & 2 & 4 
    \end{pmatrix},
    \end{equation}
and
    \begin{equation}\label{eq:vec-ii-phi}
    \quad \mathbf{c}=\mathbf{c}_{\alpha,\beta}=
    \begin{pmatrix}
        c_{1,\alpha,\beta} \\[0.5ex]
        c_{2,\alpha,\beta} \\[0.5ex]
        c_{3,\alpha,\beta} \\[0.5ex]
        c_{4,\alpha,\beta}
    \end{pmatrix},
    \quad \mathbf{v}=\mathbf{v}_{\alpha,\beta}=
    \begin{pmatrix}
        -\frac{2}{3}\gamma^2 W^M_{\alpha i i \beta}  \\[0.5ex]
        \frac{4}{3}\gamma^2 W^M_{\alpha i i \beta}  \\[0.5ex]
        -\frac{2}{3}b^2a^{-2} W^Z_{\alpha i i \beta} \\[0.5ex]
        -\frac{4}{3}b^2a^{-2} W^Z_{\alpha i i \beta}
    \end{pmatrix}.
\end{equation}

Similarly, from formula \eqref{eq:raf-fun-tilde} for $\tilde{f}_{\alpha,\beta}$ and the boundary datum \eqref{eq:bd-ii-psi} we obtain a system of type \eqref{eq:lin-syst-eq} with
 \begin{equation}\label{eq:vec-ii-psi}
    \quad \mathbf{c}=\tilde{\mathbf{c}}_{\alpha,\beta}=
    \begin{pmatrix}
        \tilde{c}_{1,\alpha,\beta} \\[0.5ex]
        \tilde{c}_{2,\alpha,\beta} \\[0.5ex]
        \tilde{c}_{3,\alpha,\beta} \\[0.5ex]
        \tilde{c}_{4,\alpha,\beta}
    \end{pmatrix},
    \quad \mathbf{v}=\tilde{\mathbf{v}}_{\alpha,\beta}=
    \begin{pmatrix}
        -\frac{1}{3}\gamma^2 W^M_{\alpha i i \alpha} \\[0.5ex]
        \frac{2}{3}\gamma^2 W^M_{\alpha i i \alpha} \\[0.5ex]
        -\frac{1}{3}b^2a^{-2} W^Z_{\alpha i i \alpha} \\[0.5ex]
        -\frac{2}{3}b^2a^{-2} W^Z_{\alpha i i \alpha}
    \end{pmatrix}.
\end{equation}

In both cases, the system has the \underline{same coefficient matrix} $A_\gamma$ and one has $\det(A_\gamma)\not=0$ $\forall \gamma\in (0,1)$. In particular, letting $\mathbf{v}=(v_1,v_2,v_3,v_4)$, one can see that the solution $\mathbf{c}=(c_1,c_2,c_3,c_4)$ to \eqref{eq:lin-syst-eq} is given by:
{\small
\begin{align}
\notag
    c_1&=a^2\frac{2v_1(1+\gamma^2+4\gamma^4)+v_2(1+\gamma^2-2\gamma^4)-\gamma^6(2v_3(4+\gamma^2+\gamma^4)+v_4(-2+\gamma^2+\gamma^4))}{2(\gamma^2-1)^3(1+4\gamma^2+\gamma^4)} \\[0.7ex]
    \notag
    c_2&=a^2\frac{-4v_1(1+\gamma^2+\gamma^4+3\gamma^6)+v_2(\gamma^2-1)(1+2\gamma^2+3\gamma^4)+\gamma^6(3(4v_3-v_4)+(4v_3+v_4)(\gamma^2+\gamma^4+\gamma^6))}{2\gamma^2(\gamma^2-1)^3(1+4\gamma^2+\gamma^4)} \\[0.7ex]
    \notag
    c_3&=a^2\frac{4 v_1(3+\gamma^2+\gamma^4+\gamma^6)-v_2(-3+\gamma^2+\gamma^4+\gamma^6)+\gamma^2((-4v_3+v_4)(1+\gamma^2+\gamma^4)-3\gamma^6(v_3+v_4))}{2\gamma^2(\gamma^2-1)^3(1+4\gamma^2+\gamma^4)} \\[0.7ex]
     \label{eq:syst-sol}c_4&=a^2\frac{-2v_1(4+\gamma^2+\gamma^4)+v_2(-2+\gamma^2+\gamma^4)+\gamma^2(2v_3(1+\gamma^2+4\gamma^4)+v_4(-1-\gamma^2+2\gamma^4))}{2\gamma^2(\gamma^2-1)^3(1+4\gamma^2+\gamma^4)}.
\end{align}}
In other words, we get an explicit formula for $f_{\alpha,\beta}$ and $\tilde{f}_{\alpha,\beta}$ in terms of $\gamma$ and the boundary datum $\mathbf{v}$ given by \eqref{eq:vec-ii-phi} and \eqref{eq:vec-ii-psi} respectively.

\bigskip

\textbullet \,\, if $i\not=j$, for $(\alpha,\beta)\in\{(j,i),(j,s),(j,t),(s,i),(t,i)\}$ one has
\begin{equation}\label{eq:bd-ij-phi-gen}
    \begin{cases*}
        f_{\alpha,\beta}(\gamma)=-\frac{1}{3}a^2\gamma^{-2}W^M_{\alpha i j \beta} \\
        f_{\alpha,\beta}^{'}(\gamma)=\frac{2}{3}a^2\gamma^{-3}W^M_{\alpha i j \beta} \\
        f_{\alpha,\beta}(1)=-\frac{1}{3}b^2 W^Z_{\alpha i j \beta} \\
        f^{'}_{\alpha,\beta}(1)=-\frac{2}{3}b^2 W^Z_{\alpha i j \beta},
    \end{cases*}
\end{equation}
while
\begin{equation}\label{eq:bd-ij-phi-psi-st}
    \begin{cases*}
        f_{s,t}(\gamma)=-\frac{1}{3}a^2\gamma^{-2}(W^M_{s i j t}+W^M_{t i j s}) \\
        f_{s,t}^{'}(\gamma)=\frac{2}{3}a^2\gamma^{-3}(W^M_{s i j t}+W^M_{t i j s}) \\
        f_{s,t}(1)=-\frac{1}{3}b^2 (W^Z_{s i j t}+W^Z_{t i j s}) \\
        f^{'}_{s,t}(1)=-\frac{2}{3}b^2 (W^Z_{s i j t}+W^Z_{t i j s}),
    \end{cases*} \qquad \text{and} \qquad \begin{cases*}
         \tilde{f}_{s,t}(\gamma)=-\frac{1}{3}a^2\gamma^{-2}W^M_{s i j s} \\
        \tilde{f}_{s,t}^{'}(\gamma)=\frac{2}{3}a^2\gamma^{-3}W^M_{s i j s} \\
        \tilde{f}_{s,t}(1)=-\frac{1}{3}b^2 W^Z_{s i j s} \\
        \tilde{f}^{'}_{s,t}(1)=-\frac{2}{3}b^2 W^Z_{s i j s}.
    \end{cases*}
\end{equation}
Then, as in the previous case (we get a system with matrix $A_\gamma$ as in \eqref{eq:agamma-matrix}), one obtains an explicit expression of $f_{\alpha,\beta}$, $\tilde{f}_{\alpha,\beta}$ (which are of the form \eqref{eq:rad-funct-eq}, \eqref{eq:raf-fun-tilde}) where the constants $c_{d,\alpha,\beta}$, $\tilde{c}_{d,\alpha,\beta}$ are as in \eqref{eq:syst-sol} and $\mathbf{v}=(v_1,v_2,v_3,v_4)$ is given by \eqref{eq:bd-ij-phi-gen}, \eqref{eq:bd-ij-phi-psi-st}. In particular, for $(\alpha,\beta)\in\{(j,i),(j,s),(j,t),(s,i),(t,i)\}$, we get
\begin{equation}\label{eq:vec-ij-phi}
    \quad \mathbf{v}_{\alpha,\beta}=
    \begin{pmatrix}
        -\frac{1}{3}\gamma^2 W^M_{\alpha i j \beta} \\[0.5ex]
        \frac{2}{3}\gamma^2 W^M_{\alpha i j \beta} \\[0.5ex]
        -\frac{1}{3}b^2a^{-2} W^Z_{\alpha i j \beta} \\[0.5ex]
        -\frac{2}{3}b^2a^{-2} W^Z_{\alpha i j \beta}
    \end{pmatrix},
\end{equation}
while 
\begin{equation}\label{eq:vec-ij-psi-hpi-st}
    \quad \mathbf{v}_{s,t}=
    \begin{pmatrix}
        -\frac{1}{3}\gamma^2 (W^M_{s i j t}+ W^M_{tijs}) \\[0.5ex]
        \frac{2}{3}\gamma^2 (W^M_{s i j t}+W^M_{tijs}) \\[0.5ex]
        -\frac{1}{3}b^2a^{-2} (W^Z_{s i j t}+ W^Z_{tijs}) \\[0.5ex]
        -\frac{2}{3}b^2a^{-2} (W^Z_{s i j t}+ W^Z_{tijs})
    \end{pmatrix}, \qquad 
    \tilde{\mathbf{v}}_{s,t}=\begin{pmatrix}
        -\frac{1}{3}\gamma^2 W^M_{s i j s} \\[0.5ex]
        \frac{2}{3}\gamma^2 W^M_{s i j s} \\[0.5ex]
        -\frac{1}{3}b^2a^{-2} W^Z_{s i j s} \\[0.5ex]
        -\frac{2}{3}b^2a^{-2} W^Z_{s i j s}
    \end{pmatrix}.
\end{equation}

\bigskip

\begin{remark}
    In both cases $i=j$ and $i\not=j$, the boundary datum $\mathbf{v}=(v_1,v_2,v_3,v_4)$ (which is given by \eqref{eq:vec-ii-phi}, \eqref{eq:vec-ii-psi} and \eqref{eq:vec-ij-phi}, \eqref{eq:vec-ij-psi-hpi-st} respectively) always satisfies the relations $v_2=-2 v_1$ and $v_4=2v_3$; using these formulae and expanding with respect to $\gamma\to 0$ (later on, we will fix the value of $b^2a^{-2}$ \emph{before} fixing $\gamma$, so this is allowed), \eqref{eq:syst-sol} turns into
    \begin{align}
    \notag
        \hat{c}_1:=a^{-2}c_1&=-6 v_1\gamma^4+ 2v_3\gamma^6+6v_1\gamma^6 +O\big(\gamma^{10}(1+b^2a^{-2})\big) \\[0.5ex]
        \notag
        \hat{c}_2:=a^{-2}c_2&=\frac{v_1}{\gamma^2}+9 v_1 \gamma^2 -3v_3\gamma^4 + O\big(\gamma^6(1+b^2a^{-2})\big) \\
        \notag
        \hat{c}_3:=a^{-2}c_3&=-\frac{3 v_1}{\gamma^2} +v_3 -27 v_1\gamma^2 +9v_3\gamma^4 +O\big(\gamma^6(1+b^2a^{-2})\big) \\
        \label{eq:const-exp-1}
        \hat{c}_4:=a^{-2}c_4&=\frac{2 v_1}{\gamma^2}+18 v_1 \gamma^2 -6 v_3\gamma^4 +O\big(\gamma^6(1+b^2a^{-2})\big).
    \end{align}
\end{remark}

\section{Comparison of Weyl energies}\label{sec:energy-balance}

We now want to compute the difference between the Weyl energy of $(X,g_X)$ defined at the end of Section \ref{sec:gluing-setup} and the sum of the Weyl energies of $(M,g_M)$ and $(Z,g_Z)$.

To begin we notice that, by definition of $(X,g_X)$ and by conformal invariance of the Weyl energy, it holds 
\begin{align}
\notag
    \w^X(g_X)-\w^M(g_M)-\w^Z(g_Z)&=\w^X(g_X)-\w^N(g_a)-\w^Z(g_b) \\
    \label{eq:energ-balance-formula}
    =\int_{B_{1+\sqrt{a}}\backslash B_{\gamma-\sqrt{a}}}&\abs{W^{g_X}}^2\,dV_{g_X}-\int_{\R^4\backslash B_{\gamma-\sqrt{a}}}\abs{W^{g_a}}^2\,dV_{g_a}-\int_{B_{1+\sqrt{a}}}\abs{W^{g_b}}^2\,dV_{g_b},
\end{align}
where the domains of each integral are to be understood with respect to the coordinates $\{x^i\}$ defined on $Z$ (and more generally on the gluing region) in Section \ref{sec:gluing-setup}.
In order to prove Theorem \ref{thm:main}, we would like to estimate the three integrals above and show that \eqref{eq:energ-balance-formula} has \emph{negative sign} for a suitable choice of the parameters $a, b, \gamma$, of the basepoints $p\in M$, $q\in Z$, and of the normal coordinate systems.

We start by computing the Weyl energies of $g_b$ and $g_a$ and then we will focus on the Weyl energy of $g_X$. We consider for now the contributions inside the regions $B_1$ and $\R^4\backslash B_\gamma$ (for $g_b$ and $g_a$ respectively) and on the annulus $B_1\backslash B_\gamma$ (for $g_X$); lastly, at the end of the section, we will show that the integrals on the remaining {\em cutoff regions} generate an higher order contribution (cf. Lemma \ref{lem:error-integrals}).

Define the $(0,2)$-symmetric tensors $H$ and $F$ by:
\begin{align}\label{eq:metr-err-def}
  H_{ij}(x):=-\frac{1}{3}W^Z_{kijl}(q)x^kx^l, \qquad  F_{ij}(x):=-\frac{1}{3}W^M_{kijl}(p)\frac{x^kx^k}{\abs{x}^4}.
\end{align}
\begin{remark}
    One has 
    \begin{equation}\label{eq:FH-biharm}
        \Delta H_{ij}\equiv 0 \quad \text{in $\R^4$}, \qquad \Delta^2 F_{ij}\equiv 0 \quad \text{in $\R^4\backslash\{0\}$, $\,\forall i,j$}; 
    \end{equation}
    the first formula follows from the trace free-property of Weyl's tensor, while the second one follows from a direct computation. Moreover, by using again the trace-free property together with the symmetries of $W^M$ and $W^Z$, it is easy to check (see the proof of Lemma \ref{lem:wdot-TT} in the Appendix) that both $F=F_{ij}$ and $H=H_{ij}$ are TT-tensors (transverse-traceless).
\end{remark}
We can now state the following:
\begin{lemma}
    One has 
    \begin{align}
    \notag
        \int_{B_1}\abs{W^{g_b}}^2\,dV_{g_b}&=-\frac{b^4}{2}\sum_{i,j,\beta}\int_{\partial B_1}H_{ij} \partial^3_{\alpha\beta\beta} H_{ij}\nu^\alpha\,d\sigma \\
    \label{eq:weyl-en-g_b-exp}
    +b^4\sum_{i,j,\beta}&\int_{\partial B_1}\Big[\big(\partial^2_{ij} H_{\alpha \beta}-2\partial^2_{\beta j} H_{i\alpha}+\partial^2_{\alpha\beta}H_{ij}\big)\partial_\beta H_{ij}-\frac{1}{2}\partial^2_{\beta \beta} H_{ij} \partial_\alpha H_{ij}\Big]\nu^\alpha \,d\sigma+ O(b^6),
    \end{align}
    \begin{align}
    \notag
        \int_{\R^4\backslash B_\gamma}\abs{W^{g_a}}^2\,dV_{g_a}&=-\frac{a^4}{2}\sum_{i,j,\beta}\int_{\partial B_{\gamma}}F_{ij} \partial^3_{\alpha\beta\beta} F_{ij}\nu^\alpha\,d\sigma \\
    \label{eq:weyl-en-g_a-exp}
    +a^4\sum_{i,j,\beta}&\int_{\partial B_{\gamma}}\Big[\big(\partial^2_{ij} F_{\alpha \beta}-2\partial^2_{\beta j} F_{i\alpha}+\partial^2_{\alpha\beta}F_{ij}\big)\partial_\beta F_{ij}-\frac{1}{2}\partial^2_{\beta \beta} F_{ij} \partial_\alpha F_{ij}\Big]\nu^\alpha \,d\sigma+ O(a^6).
    \end{align}
\end{lemma}
\begin{proof}
    By \eqref{eq:g_b-exp}, \eqref{eq:err-g_b} and \eqref{eq:g_a-exp-fin}, \eqref{eq:err-g_a-fin} we immediately see that, letting $s:=b^2$, $t:=a^2$, then $\frac{d}{ds}\big\rvert_{s=0}(g_b)_{ij}=H_{ij}$ and $\frac{d}{dt}\big\rvert_{t=0}(g_a)_{ij}=F_{ij}$ respectively. At this point, \eqref{eq:weyl-en-g_b-exp} follows from \eqref{eq:weyl-exp-bilap-fin} (take $\Omega=B_1$) once we recall \eqref{eq:FH-biharm} and the fact that $H$ is TT. 

The same argument also works for \eqref{eq:weyl-en-g_a-exp}: indeed, even if the domain is unbounded this time, by definition of $F_{ij}$ it holds
\begin{equation*}
    \big\lvert\nabla^2F_{ij}(x)\big\rvert\big\lvert\nabla^k F_{ij}(x)\big\rvert=O(\abs{x}^{-6-k}) \quad \text{as $\abs{x}\to\infty$,}
\end{equation*}
so that we can still perform all the integration by parts leading to \eqref{eq:weyl-exp-bilap-fin}.
\end{proof}

\bigskip

We can now turn our attention to the computation of Weyl's energy for the interpolating metric $g_X$ in $B_1\backslash B_\gamma$, where it holds $(g_X)_{ij}=w_{ij}$ by construction.

\begin{remark}
    From now on we assume that $b=\lambda a$, where $\lambda >0$ will be specified later; notice however that, in the end, the constants $\lambda, a,\gamma$ will be chosen such that $0<a\ll\gamma\ll \lambda^{-1}<1$, so we will implicitly assume this while performing the various expansions.
\end{remark}

By the results in Section \ref{sec:biharm-interp}, we know that each $w_{ij}$ writes as
\begin{equation*}
    w_{ij}(x)=\delta_{ij} +\sum_{\alpha} f_{\alpha,i,j}(r)\phi_\alpha (\theta),
\end{equation*}
where $r=\abs{x}$ and $\theta=\frac{x}{\abs{x}}$, see \eqref{eq:sol-ii-form} and \eqref{eq:sol-ij-form}. Moreover, each $f_\alpha$ is of the form
\begin{equation*}
    f(r)=\frac{c_{1}}{r^4}+\frac{c_{2}}{r^2}+ c_{3} r^2 +c_{4} r^4,
\end{equation*}
where $c_1,c_2,c_3,c_4$ are given by \eqref{eq:const-exp-1}. By \eqref{eq:vec-ii-phi}, \eqref{eq:vec-ii-psi}, \eqref{eq:vec-ij-phi} and \eqref{eq:vec-ij-psi-hpi-st} we see that one always has $v_1=O(\gamma^2)$ and $v_3=O(b^2a^{-2})=O(\lambda^2)$, so, if we let $t:=a^2$ and consider $w=(w_{ij})=(w_{ij})(t)$ as a family of metrics on $B_1\backslash B_\gamma$ depending on $t$, then $w_{ij}(0)=\delta_{ij}$ and, letting 
\begin{equation*}
    \dot{w}_{ij}:=\frac{d}{dt}\biggr\rvert_{t=0}w_{ij}(t),
\end{equation*}
then $\dot{w}_{ij}$ takes the general form
\begin{equation}\label{eq:w-wddot-relation}
    \dot{w}_{ij}=a^{-2}\sum_{\alpha} f_{\alpha,i,j}(r)\phi_\alpha (\theta)=a^{-2}(w_{ij}-\delta_{ij}).
\end{equation}
\begin{remark}
  \eqref{eq:w-wddot-relation} is in fact an equality. In other words, $w_{ij}$ is precisely of the form \emph{Euclidean metric $+$ linear error in $t$} because of our choice of (exact) boundary data in \eqref{eq:interp-system}.
\end{remark}

We also have the following fact: 
\begin{lemma}\label{lem:wdot-TT}
    $\dot{w}$ is a TT tensor.
\end{lemma}
The proof is a long but trivial check and is postponed to the Appendix. As a consequence, we can apply \eqref{eq:weyl-exp-bilap-fin} to the metric $w(t)$ in $\Omega=B_1\backslash B_\gamma$: recalling \eqref{eq:w-wddot-relation} and that $w$ is harmonic by construction (see \eqref{eq:interp-system}), we get 
\begin{align}
\notag
    \w(w)&:=\int_{B_1\backslash B_\gamma}\abs{W_{g_X}}^2\,dV_{g_X}=-\frac{a^4}{2}\sum_{i,j,\beta}\int_{\partial(B_1\backslash B_\gamma)}\dot{w}_{ij} \partial^3_{\alpha\beta\beta} \dot{w}_{ij}\nu^\alpha\,d\sigma \\
    \label{eq:weyl-interp-exp}
    &+a^4\sum_{i,j,\beta}\int_{\partial(B_1\backslash B_\gamma)}\Big[\big(\partial^2_{ij} \dot{w}_{\alpha \beta}-2\partial^2_{\beta j} \dot{w}_{i\alpha}+\partial^2_{\alpha\beta}\dot{w}_{ij}\big)\partial_\beta \dot{w}_{ij}-\frac{1}{2}\partial^2_{\beta \beta} \dot{w}_{ij} \partial_\alpha \dot{w}_{ij}\Big]\nu^\alpha \,d\sigma+ O(a^6).
\end{align}

We thus have an expression involving products of mixed derivatives of $\dot{w}_{ij}$; in order to estimate this quantity, we first look at the general formula of $\dot{w}_{ij}$ and try to isolate the type of terms which generate the biggest contributions with respect to the parameter $\gamma$.

By definition of $\dot{w}_{ij}$ one sees that, for $x\in B_1\backslash B_\gamma$,
\begin{align}
\notag
    \dot{w}_{ij}&\simeq\hat{c}_1\abs{x}^{-4}+ \hat{c}_2 \abs{x}^{-2} +\hat{c}_3 \abs{x}^2+\hat{c}_4 \abs{x}^4, \\
    \notag
    \partial_\alpha\dot{w}_{ij}&\simeq\hat{c}_1\abs{x}^{-5}+ \hat{c}_2 \abs{x}^{-3} +\hat{c}_3 \abs{x}+\hat{c}_4 \abs{x}^3,\\
    \notag
    \partial^2_{\alpha\beta}\dot{w}_{ij}&\simeq\hat{c}_1\abs{x}^{-6}+ \hat{c}_2 \abs{x}^{-4} +\hat{c}_3 +\hat{c}_4 \abs{x}^2, \\
    \label{eq:wdot-main-order}
    \partial^3_{\alpha\beta\tau}\dot{w}_{ij}&\simeq\hat{c}_1\abs{x}^{-7}+ \hat{c}_2 \abs{x}^{-5} +\hat{c}_3 \abs{x}^{-1}+\hat{c}_4 \abs{x},
\end{align}
where $\hat{c}_1,\dots,\hat{c}_4$ are given by \eqref{eq:const-exp-1}.

\medskip

\textbullet \,\, If $\abs{x}=\gamma$ (i.e. we are considering the integrals on $\partial B_\gamma$), then from \eqref{eq:wdot-main-order} and \eqref{eq:const-exp-1} (recall that $v_1=O(\gamma^2)$ and $v_3=O(\lambda^2)$) we see that the main contribution comes from the $c_2/\abs{x}^2$-terms, which  generate integrals of order $\gamma^{-4}$ in \eqref{eq:weyl-interp-exp}. Indeed, for $1\ll\gamma\ll\lambda^{-1}$ one has
\begin{equation}\label{eq:chat-exp-mainorder}
    \hat{c}_1=O(\gamma^6(1+\lambda^2)), \qquad \hat{c}_2=O(1),\qquad \hat{c}_3=O(1+\lambda^2), \qquad \hat{c}_4=O(1),
\end{equation}
so that 
\begin{equation*}
    \dot{w}_{ij}=O(\abs{x}^{-2})=O(\gamma^{-2}), \qquad \partial_\alpha\dot{w}_{ij}=O(\gamma^{-3}),\qquad \partial^2_{\alpha\beta}\dot{w}_{ij}=O(\gamma^{-4}),\qquad \partial^3_{\alpha\beta\tau}\dot{w}_{ij}=O(\gamma^{-5}),
\end{equation*}
and therefore
\begin{align*}
    \int_{\partial B_\gamma}\partial_{\alpha}\dot{w}_{ij}\partial^2_{\beta \tau}\dot{w}_{ij}\,d\sigma=O(\gamma^{-4})=\int_{\partial B_\gamma}\dot{w}_{ij}\partial^3_{\alpha\beta\tau}\dot{w}_{ij}\,d\sigma.
\end{align*}
We now want to write an explicit expression for the main term.

\underline{If $i=j$}, then, from \eqref{eq:w-wddot-relation}, \eqref{eq:sol-ii-form}, \eqref{eq:sph-harm-def}, \eqref{eq:rad-funct-eq}, \eqref{eq:raf-fun-tilde} and the above discussion we get
\begin{align*}
    a^2\dot{w}_{ii}(x)&=f_{s,t}(\abs{x})\phi_{s,t}+f_{s,u}(\abs{x})\phi_{s,u}+f_{t,u}(\abs{x})\phi_{t,u}+\tilde{f}_{s,u}(\abs{x})\psi_{s,u}+\tilde{f}_{t,u}(\abs{x})\psi_{t,u} \\
    &=\frac{{c}_{2,s,t}}{\abs{x}^2}\frac{x_sx_t}{\abs{x}^2}+\frac{{c}_{2,s,u}}{\abs{x}^2}\frac{x_sx_u}{\abs{x}^2}+\frac{{c}_{2,t,u}}{\abs{x}^2}\frac{x_tx_u}{\abs{x}^2}+\frac{\tilde{c}_{2,s,u}}{\abs{x}^2}\frac{x_s^2-x_u^2}{\abs{x}^2}+\frac{\tilde{c}_{2,t,u}}{\abs{x}^2}\frac{x_t^2-x_u^2}{\abs{x}^2}+ \text{h.o.t.}
\end{align*}
By \eqref{eq:const-exp-1}, we have $c_2=a^2(v_1/\gamma^2+O(\gamma^4(1+\lambda^2)))$, so, taking from \eqref{eq:vec-ii-phi}, \eqref{eq:vec-ii-psi} the values of $v_1$ for each term, one gets 
\begin{align*}
    \dot{w}_{ii}=-\frac{2}{3}W^M_{siit}\frac{x_sx_t}{\abs{x}^4}-\frac{2}{3}W^M_{siiu}\frac{x_sx_u}{\abs{x}^4}-\frac{2}{3}W^M_{tiiu}\frac{x_tx_u}{\abs{x}^4}-\frac{1}{3}W^M_{siis}\frac{x_s^2-x_u^2}{\abs{x}^4}-\frac{1}{3}W^M_{tiit}\frac{x_t^2-x_u^2}{\abs{x}^4}+\text{h.o.t.}
\end{align*}
Using now the trace-free property of Weyl's tensor we finally obtain
\begin{equation}\label{eq:mainterm-ii}
    \dot{w}_{ii}=-\frac{1}{3}W^M_{kiil}\frac{x^kx^l}{\abs{x}^4}+\text{h.o.t.}=F_{ii}(x)+\text{h.o.t.}
\end{equation}

\underline{If $i\not=j$}, we can proceed in the exact same way using instead \eqref{eq:sol-ij-form}, \eqref{eq:vec-ij-phi} and \eqref{eq:vec-ij-psi-hpi-st}; in the end we find 
\begin{equation}\label{eq:mainterm-ij}
    \dot{w}_{ij}=-\frac{1}{3}W^M_{kijl}\frac{x^kx^l}{\abs{x}^4}+\text{h.o.t.}=F_{ij}(x)+\text{h.o.t.}
\end{equation}

By \eqref{eq:mainterm-ii}, \eqref{eq:mainterm-ij} we see that, at the main order, the Weyl energy of $w$ in $B_1\backslash B_\gamma$ coincides with the Weyl energy of $g_a$ in $\R^4\backslash B_\gamma$ (just compare \eqref{eq:weyl-interp-exp} to \eqref{eq:weyl-en-g_a-exp}); as a consequence, the two contributions cancel out in the energy balance \eqref{eq:energ-balance-formula}.

\medskip

\textbullet \,\, If $\abs{x}=1$ instead, then by \eqref{eq:wdot-main-order} and \eqref{eq:const-exp-1} we see that the main contribution comes from the products of $c_3\abs{x}^2$-terms, which will generate integrals of order $\lambda^{4}$ in \eqref{eq:weyl-interp-exp}. Indeed, by \eqref{eq:wdot-main-order} and \eqref{eq:chat-exp-mainorder} we have in this case
\begin{equation*}
    \dot{w}_{ij}=O(\lambda^2), \qquad \partial_\alpha\dot{w}_{ij}=O(\lambda^2),\qquad \partial^2_{\alpha\beta}\dot{w}_{ij}=O(\lambda^2),\qquad \partial^3_{\alpha\beta\tau}\dot{w}_{ij}=O(\lambda^2),
\end{equation*}
therefore
\begin{align*}
    \int_{\partial B_1}\partial_{\alpha}\dot{w}_{ij}\partial^2_{\beta \tau}\dot{w}_{ij}\,d\sigma=O(\lambda^4)=\int_{\partial B_1}\dot{w}_{ij}\partial^3_{\alpha\beta\tau}\dot{w}_{ij}\,d\sigma.
\end{align*}
As in the previous case, we can write down explicitly the main order term of $\dot{w}_{ij}$ by using \eqref{eq:w-wddot-relation}, \eqref{eq:sol-ii-form}, \eqref{eq:sol-ij-form}, \eqref{eq:rad-funct-eq}, \eqref{eq:raf-fun-tilde}, \eqref{eq:const-exp-1} and the explicit values of $v_3$ given by the formulae in Section \ref{sec:biharm-interp}. We end up finding that
\begin{equation}\label{eq:b^4term-interp}
    \dot{w}_{ij}=-\frac{1}{3}\lambda^2W^Z_{kijl} x^kx^l + \text{h.o.t.}=\lambda^2 H_{ij}(x)+\text{h.o.t.}
\end{equation}
In particular, we see that, at order $a^4\lambda^4=b^4$, the Weyl energy of $w$ in $B_1\backslash B_\gamma$ coincides with the Weyl energy of $g_b$ in $B_1$ (just compare \eqref{eq:weyl-interp-exp} to \eqref{eq:weyl-en-g_b-exp}), therefore the two contributions cancel out in the energy balance \eqref{eq:energ-balance-formula}.

\medskip

    As a consequence of the previous computations, in order to get the sign of energy balance we need to look at \emph{higher order terms} in the expansion of $\dot{w}_{ij}$ at $\abs{x}=1$ and $\abs{x}=\gamma$ respectively. This will be carried out in the next subsections.

\subsection{Higher order terms in energy balance: inner boundary component}

We now want to compute the higher order contributions coming from the integrals over $\partial B_\gamma$ of the Weyl energy expansion \eqref{eq:weyl-interp-exp}. By looking at \eqref{eq:wdot-main-order}, \eqref{eq:chat-exp-mainorder}, \eqref{eq:const-exp-1}, one sees that such contributions come from products of derivatives of $\hat{c}_1\abs{x}^{-4}$, $\hat{c}_2\abs{x}^{-2}$ and $\hat{c}_3\abs{x}^2$-terms \emph{times} derivatives of $\hat{c}_2\abs{x}^{-2}$-terms, which generate integrals of order $\lambda^2$ and integrals of order $1$. Next, we have contributions of higher order $\gamma^2\lambda^2$. To be more precise, if we look at the definition of $\dot{w}_{ij}$, then, using \eqref{eq:sph-harm-def}, \eqref{eq:sol-ii-form}, \eqref{eq:sol-ij-form}, \eqref{eq:const-exp-1} and the values of $v_1,v_3$ obtained in \eqref{eq:vec-ii-phi}, \eqref{eq:vec-ii-psi}, \eqref{eq:vec-ij-phi}, \eqref{eq:vec-ij-psi-hpi-st} (look also at the proof of Lemma \ref{lem:wdot-TT} in the Appendix), we can write $\dot{w}_{ij}$ as follows in $\abs{x}=\gamma$:
\begin{align}
\notag
    \dot{w}_{ij}&=\frac{1}{\abs{x}^4}\Big(2\gamma^6 W^M_{kijl}\frac{x^kx^l}{\abs{x}^2}-\frac{2}{3}\lambda^2\gamma^6 W^Z_{kijl}\frac{x^kx^l}{\abs{x}^2}\Big)+ O^{(3)}(\gamma^8(1+\lambda^2)\abs{x}^{-4}) \\
    \notag
    &\quad+\frac{1}{\abs{x}^2}\Big(-\frac{1}{3}W^M_{kijl}\frac{x^kx^l}{\abs{x}^2}-3\gamma^4 W^M_{kijl}\frac{x^kx^l}{\abs{x}^2}+\lambda^2\gamma^4W^Z_{kijl}\frac{x^kx^l}{\abs{x}^2}\Big)+O^{(3)}(\gamma^6(1+\lambda^2)\abs{x}^{-2}) \\
    \notag
    &\quad+\abs{x}^2\Big(W^M_{kijl}\frac{x^kx^l}{\abs{x}^2}-\frac{1}{3}\lambda^2 W^Z_{kijl}\frac{x^kx^l}{\abs{x}^2}\Big)+ O^{(3)}(\gamma^4(1+\lambda^2)\abs{x}^2) +O^{(3)}((1+\gamma^4\lambda^2)\abs{x}^4) \\
    \label{eq:wdot-gamma-exp}
    &=-\frac{1}{3}W^M_{kijl}\frac{x^kx^l}{\abs{x}^4}+\Big(\frac{2\gamma^6}{\abs{x}^6}-\frac{3\gamma^4}{\abs{x}^4}+1\Big)W^M_{kijl}x^kx^l 
     +\lambda^2\Big(-\frac{2\gamma^6}{3\abs{x}^6}+\frac{\gamma^4}{\abs{x}^4}-\frac{1}{3}\Big)W^Z_{kijl}x^kx^l +\Psi_{\lambda,\gamma}(x),
\end{align}
where $\Psi_{\lambda,\gamma}(x)$ is a higher order error term satisfying $\abs{\nabla^k \Psi_{\lambda,\gamma}(x)}\leq C\lambda^2\gamma^{4-k}$ for $\abs{x}=\gamma$ and $k=0,1,2,3$. Looking at \eqref{eq:wdot-gamma-exp}, we clearly see that the main term is given by $-\frac{1}{3}W^M_{kijl}\frac{x^kx^l}{\abs{x}^4}=F_{ij}(x)$ as already pointed out in \eqref{eq:mainterm-ii}, \eqref{eq:mainterm-ij}. It follows that the biggest terms in the expansion of integrals on $\partial B_\gamma$ of \eqref{eq:weyl-interp-exp} will be given by products of derivatives of such term and it will be equal to \eqref{eq:weyl-en-g_a-exp}. On the other hand, the next biggest term on $\partial B_\gamma$ is given by mixed terms and is of order $\lambda^2$ as in the following example:

\begin{example}
    Consider for instance the integral
    \begin{equation*}
        \int_{\partial B_\gamma} \partial^2_{\alpha\beta}\dot{w}_{ij}\partial_\beta\dot{w}_{ij}\nu^\alpha\,d\sigma,
    \end{equation*}
    which is one of the terms in \eqref{eq:weyl-interp-exp}. Then $\dot{w}_{ij}$ expands as in \eqref{eq:wdot-gamma-exp} and one has
    \begin{align}
    \notag
        \int_{\partial B_\gamma} \partial^2_{\alpha\beta}\dot{w}_{ij}\partial_\beta\dot{w}_{ij}\nu^\alpha\,d\sigma&=\underbrace{\int_{\partial B_\gamma} \partial^2_{\alpha\beta}\Big(-\frac{1}{3}W^M_{kijl}\frac{x^kx^l}{\abs{x}^4}\Big)\partial_\beta\Big(-\frac{1}{3}W^M_{kijl}\frac{x^kx^l}{\abs{x}^4}\Big)\nu^\alpha}_{=C_1\gamma^{-4}} \\
        \notag
        &+\underbrace{\int_{\partial B_\gamma} \partial^2_{\alpha\beta}\Big(-\frac{1}{3}W^M_{kijl}\frac{x^kx^l}{\abs{x}^4}\Big)\partial_\beta\Big(\lambda^2\Big(-\frac{2\gamma^6}{3\abs{x}^6}+\frac{\gamma^4}{\abs{x}^4}-\frac{1}{3}\Big)W^Z_{kijl}x^kx^l\Big)\nu^\alpha}_{=C_2\lambda^2} \\
        \notag
        &+\underbrace{\int_{\partial B_\gamma}\partial^2_{\alpha\beta}\Big(\lambda^2\Big(-\frac{2\gamma^6}{3\abs{x}^6}+\frac{\gamma^4}{\abs{x}^4}-\frac{1}{3}\Big)W^Z_{kijl}x^kx^l\Big)\partial_\beta\Big(-\frac{1}{3}W^M_{kijl}\frac{x^kx^l}{\abs{x}^4}\Big)\nu^\alpha}_{=C_3\lambda^2} \\
        \notag
        &+\underbrace{\int_{\partial B_\gamma}\partial^2_{\alpha\beta}\Big(-\frac{1}{3}W^M_{kijl}\frac{x^kx^l}{\abs{x}^4}\Big)\partial_\beta\Big(\Big(\frac{2\gamma^6}{\abs{x}^6}-\frac{3\gamma^4}{\abs{x}^4}+1\Big)W^M_{kijl}x^kx^l\Big)\nu^\alpha}_{=C_4} \\
        \label{eq:example-integrals}
        &+\underbrace{\int_{\partial B_\gamma}\partial^2_{\alpha\beta}\Big(\Big(\frac{2\gamma^6}{\abs{x}^6}-\frac{3\gamma^4}{\abs{x}^4}+1\Big)W^M_{kijl}x^kx^l\Big)\partial_\beta\Big(-\frac{1}{3}W^M_{kijl}\frac{x^kx^l}{\abs{x}^4}\Big)\nu^\alpha}_{=C_5}+O(\gamma^2\lambda^2),
    \end{align}
    where $C_1, \dots,C_5$ are constants which depend on $i,j,\alpha,\beta$ and $W^M$, $W^Z$ (but not on $\gamma$ and $\lambda$!).
    Similarly, if we consider the other type of integral in \eqref{eq:weyl-interp-exp},
    \begin{equation*}
        \int_{\partial B_\gamma} \dot{w}_{ij}\partial^3_{\alpha\beta\beta}\dot{w}_{ij}\nu^\alpha\,d\sigma,
    \end{equation*}
    we obtain an expression which is similar to the one above (we only need to change the derivatives).
\end{example}

\begin{remark}
    As it is clear from \eqref{eq:example-integrals}, the mixed terms of order $\lambda^2$ depend upon the \emph{interaction} between the two Weyl tensors $W^M(p)$ and $W^Z(q)$ evaluated at their respective basepoints with respect to our choice of coordinates. On the other hand, the terms of order $1$ only depend upon $W^M(p)$.
\end{remark}

We now want to explicitly compute the term of order $\lambda^2$ for the inner integral
\begin{align}
\notag
    \Phi_\gamma(w):&=\frac{1}{2}\sum_{i,j,\beta}\int_{\partial B_\gamma}\dot{w}_{ij} \partial^3_{\alpha\beta\beta} \dot{w}_{ij}\nu^\alpha\,d\sigma \\
    \label{eq:weyl-inner-exp}
    &-\sum_{i,j,\beta}\int_{\partial B_\gamma}\Big[\big(\partial^2_{ij} \dot{w}_{\alpha \beta}-2\partial^2_{\beta j} \dot{w}_{i\alpha}+\partial^2_{\alpha\beta}\dot{w}_{ij}\big)\partial_\beta \dot{w}_{ij}-\frac{1}{2}\partial^2_{\beta \beta} \dot{w}_{ij} \partial_\alpha \dot{w}_{ij}\Big]\nu^\alpha \,d\sigma;
\end{align}
here we are assuming that the unit normal $\nu=(\nu^1,\dots,\nu^4)$ is \emph{outward pointing}.

To begin, we derive an expression for a generic term in the second line of \eqref{eq:weyl-inner-exp}:
\begin{lemma}\label{lemma:tecnical}
    It holds
    \begin{align}
    \notag
        \int_{\partial B_\gamma}\partial_{\alpha} \dot{w}_{ij}&\partial^2_{\beta\tau}\dot{w}_{st}\nu^\eta\,d\sigma=\int_{\partial B_\gamma}\partial_\alpha F_{ij} \partial^2_{\beta\tau} F_{st}\nu^\eta\,d\sigma \\
        \label{eq:gammaint-1-est}
        &+\frac{8}{3}\lambda^2\int_{\Sp^3}W^M_{\mu i j \nu}W^Z_{kstl}(\delta^\mu_\alpha x^\nu+x^\mu\delta^\nu_\alpha -4 x^\mu x^\nu x_\alpha)x^kx^lx_\beta x_\tau x^\eta\,d\sigma +C +O(\gamma^2\lambda^2),
    \end{align}
    where $F_{ij}$ is given by \eqref{eq:metr-err-def} and $C=C(\alpha,\beta,\tau,i,j,s,t,\eta, W^M)$.
\end{lemma}

\begin{proof}
We start by computing the derivatives of the {\em interacting terms} of $\dot{w}_{ij}$ (cf. \eqref{eq:wdot-gamma-exp}): to begin, one has
\begin{align*}
    \partial_\alpha\bigg[\Big(-\frac{2\gamma^6}{3\abs{x}^6}&+\frac{\gamma^4}{\abs{x}^4}-\frac{1}{3}\Big)W^Z_{kijl}x^kx^l\bigg] \\
    &=\Big(4\gamma^6\frac{x_\alpha}{\abs{x}^8}-4\gamma^4\frac{x_\alpha}{\abs{x}^6}\Big)W^Z_{kijl}x^kx^l+\Big(-\frac{2\gamma^6}{3\abs{x}^6}+\frac{\gamma^4}{\abs{x}^4}-\frac{1}{3}\Big)W^Z_{kijl}(\delta^k_\alpha x^l+x^k\delta^l_\alpha).
\end{align*}
\begin{remark}\label{rem:comput}
We notice that
\begin{equation}\label{eq:der1-gamma}
\Big(-\frac{2\gamma^6}{3\abs{x}^6}+\frac{\gamma^4}{\abs{x}^4}-\frac{1}{3}\Big)\biggr\rvert_{\abs{x}=\gamma}=0, \quad \text{and} \quad
    \partial_\alpha\bigg[\Big(-\frac{2\gamma^6}{3\abs{x}^6}+\frac{\gamma^4}{\abs{x}^4}-\frac{1}{3}\Big)W^Z_{kijl}x^kx^l\bigg]\biggr\rvert_{\abs{x}=\gamma}=0.
\end{equation}
    As a consequence, we do not need to compute $ \partial^2_{\alpha\beta}\Big(-\frac{1}{3}W^M_{kijl}\frac{x^kx^l}{\abs{x}^4}\Big)$ and $ \partial^3_{\alpha\beta\beta}\Big(-\frac{1}{3}W^M_{kijl}\frac{x^kx^l}{\abs{x}^4}\Big)$ since these quantities will always be multiplied by zero in the end, cf. \eqref{eq:example-integrals}.
\end{remark}
We then compute
\begin{align*}
    \partial^2_{\alpha\beta}\bigg[\Big(-\frac{2\gamma^6}{3\abs{x}^6}&+\frac{\gamma^4}{\abs{x}^4}-\frac{1}{3}\Big)W^Z_{kijl}x^kx^l\bigg]=\Big(4\gamma^6\frac{\delta_{\alpha\beta}}{\abs{x}^8}-32\gamma^6\frac{x_\alpha x_\beta}{\abs{x}^{10}}-4\gamma^4\frac{\delta_{\alpha\beta}}{\abs{x}^6}+24\gamma^4\frac{x_\alpha x_\beta}{\abs{x}^8}\Big)W^Z_{kijl}x^kx^l \\
    &+\Big(4\gamma^6\frac{x_\alpha}{\abs{x}^8}-4\gamma^4\frac{x_\alpha}{\abs{x}^6}\Big)W^Z_{kijl}(\delta^k_\beta x^l+x^k\delta^l_\beta)+\Big(4\gamma^6\frac{x_\beta}{\abs{x}^8}-4\gamma^4\frac{x_\beta}{\abs{x}^6}\Big)W^Z_{kijl}(\delta^k_\alpha x^l+x^k\delta^l_\alpha) \\
    &+\Big(-\frac{2\gamma^6}{3\abs{x}^6}+\frac{\gamma^4}{\abs{x}^4}-\frac{1}{3}\Big)W^Z_{kijl}(\delta^k_\alpha \delta^l_\beta+\delta^k_\beta\delta^l_\alpha),
\end{align*}
so that
\begin{align}\label{eq:der2-gamma}
    \partial^2_{\alpha\beta}\bigg[\Big(-\frac{2\gamma^6}{3\abs{x}^6}&+\frac{\gamma^4}{\abs{x}^4}-\frac{1}{3}\Big)W^Z_{kijl}x^kx^l\bigg]\biggr\rvert_{\abs{x}=\gamma}=-8\frac{x_\alpha x_\beta}{\gamma^4}W^Z_{kijl}x^kx^l.
\end{align}
Since we will need it later on, we also compute the third derivative: 
\begin{align}
\notag
    \partial^3_{\alpha\beta\beta}\bigg[\Big(-\frac{2\gamma^6}{3\abs{x}^6}&+\frac{\gamma^4}{\abs{x}^4}-\frac{1}{3}\Big)W^Z_{kijl}x^kx^l\bigg]\biggr\rvert_{\abs{x}=\gamma}=\Big(-16\frac{x_\beta\delta_{\alpha\beta}}{\gamma^4}-8\frac{x_\alpha}{\gamma^4}+128\frac{x_\alpha(x_\beta)^2}{\gamma^6}\Big)W^Z_{kijl}x^kx^l \\
    \label{eq:der3-gamma}
    &-16\frac{x_\alpha x_\beta}{\gamma^4}W^Z_{kijl}(\delta_\beta^{ k}x^l+x^k\delta_\beta^{ l})-8\frac{(x_\beta)^2}{\gamma^4}W^Z_{kijl}(\delta_\alpha^{ k}x^l+x^k\delta_\alpha^{ l}).
\end{align}
Finally, one has
\begin{align}\label{eq:der4-gamma}
    \partial_\alpha\Big(-\frac{1}{3}W^M_{kijl}\frac{x^kx^l}{\abs{x}^4}\Big)=-\frac{1}{3}W^M_{kijl}\Big(\frac{\delta^k_\alpha x^l+x^k\delta^l_\alpha}{\abs{x}^4}-4\frac{x^kx^lx_\alpha}{\abs{x}^6}\Big).
\end{align}

Arguing now as in \eqref{eq:example-integrals} and using \eqref{eq:wdot-gamma-exp}, \eqref{eq:der1-gamma}, \eqref{eq:der2-gamma} and \eqref{eq:der4-gamma} one has
\begin{align*}
    \int_{\partial B_\gamma}\partial_{\alpha} \dot{w}_{ij}\partial^2_{\beta\tau}\dot{w}_{st}\nu^\eta\,d\sigma&=\int_{\partial B_\gamma}\partial_\alpha\Big(-\frac{1}{3}W^M_{\mu ij\nu}\frac{x^\mu x^\nu}{\abs{x}^4}\Big) \partial^2_{\beta\tau}\Big(-\frac{1}{3}W^M_{kstl}\frac{x^kx^l}{\abs{x}^4}\Big)\nu^\eta\,d\sigma \\
    +\lambda^2\int_{\partial B_\gamma}\partial_\alpha&\Big(-\frac{1}{3}W^M_{\mu ij\nu}\frac{x^\mu x^\nu}{\abs{x}^4}\Big)\partial^2_{\beta \tau}\bigg[\Big(-\frac{2\gamma^6}{3\abs{x}^6}+\frac{\gamma^4}{\abs{x}^4}-\frac{1}{3}\Big)W^Z_{kstl}x^kx^l\bigg]\nu^\eta\,d\sigma+C +O(\lambda^2\gamma^2) \\
    &=\int_{\partial B_\gamma}\partial_\alpha F_{ij} \partial^2_{\beta\tau} F_{st}\nu^\eta\,d\sigma \\
    +\frac{8}{3}\lambda^2\int_{\partial B_\gamma}&W^M_{\mu ij \nu} W^Z_{kstl}\Big(\frac{\delta^\mu_\alpha x^\nu+x^\mu \delta^\nu_\alpha}{\gamma^4}-4\frac{x^\mu x^\nu x_\alpha}{\gamma^6}\Big)\gamma^{-4} x_\beta x_\tau x^k x^l \nu^\eta \,d\sigma+C+O(\lambda^2\gamma^2),
\end{align*}
and now a change of variables in the second integral gives us \eqref{eq:gammaint-1-est}.
\end{proof}

We can now expand \eqref{eq:weyl-inner-exp}:
\begin{lemma}\label{lem:est-gamma}
    Let $\Phi_\gamma(w)$ be defined as in \eqref{eq:weyl-inner-exp}; then:
    \begin{align*}
        \Phi_\gamma(w)=a^{-4}\int_{\R^4\backslash B_\gamma}\abs{W^{g_a}}^2\,dV_{g_a}-\frac{2}{9}\pi^2\lambda^2 W^M(p)\stell W^Z(q)+ C+O(\lambda^2\gamma^2)+O(a^2),
    \end{align*}
    where $C=C(W^M(p))$ and $W^M(p)\stell W^Z(q)$ is defined in \eqref{eq:weyl-star-weyl} (here with $M$ in place of $\overbar{M}$).
\end{lemma}
\begin{proof}
    As in the proof of Lemma \ref{lemma:tecnical}, we can use \eqref{eq:wdot-gamma-exp}, \eqref{eq:der1-gamma} and \eqref{eq:der3-gamma} and argue as in \eqref{eq:example-integrals} to obtain
    \begin{align}
    \notag
        \int_{\partial B_\gamma}\dot{w}_{ij}\partial^3_{\alpha \beta \beta}&\dot{w}_{ij}\nu^\alpha\,d\sigma=\int_{\partial B_\gamma}F_{ij}\partial^3_{\alpha\beta \beta}F_{ij}\nu^\alpha \,d\sigma \\
        \notag
        &\quad+\lambda^2\int_{\partial B_\gamma}F_{ij}\partial^3_{\alpha \beta \beta}\bigg[\Big(-\frac{2\gamma^6}{3\abs{x}^6}+\frac{\gamma^4}{\abs{x}^4}-\frac{1}{3}\Big)W^Z_{kijl}x^kx^l\bigg]\nu^\alpha \,d\sigma +C+O(\lambda^2\gamma^2) \\
        \notag
        &=\int_{\partial B_\gamma}F_{ij}\partial^3_{\alpha\beta \beta}F_{ij}\nu^\alpha \,d\sigma+\lambda^2\int_{\partial B_\gamma} F_{ij}\bigg[\Big(-16\frac{x_\beta\delta_{\alpha\beta}}{\gamma^4}-8\frac{x_\alpha}{\gamma^4}+128\frac{x_\alpha(x_\beta)^2}{\gamma^6}\Big)W^Z_{kijl}x^kx^l \\
        \notag
        &\quad-16\frac{x_\alpha x_\beta}{\gamma^4}W^Z_{kijl}(\delta_\beta^{ k}x^l+x^k\delta_\beta^{ l})-8\frac{(x_\beta)^2}{\gamma^4}W^Z_{kijl}(\delta_\alpha^{ k}x^l+x^k\delta_\alpha^{ l})\bigg]\nu^\alpha\,d\sigma+C+O(\lambda^2\gamma^2) \\
        \notag
        &=\int_{\partial B_\gamma}F_{ij}\partial^3_{\alpha\beta \beta}F_{ij}\nu^\alpha \,d\sigma+\frac{8}{3}\lambda^2\int_{\partial B_\gamma} W^M_{\mu ij\nu}W^Z_{kijl}\frac{x^\mu x^\nu}{\gamma^8}\bigg[\big(2x_\beta \delta_{\alpha \beta}+ x_\alpha-16\frac{x_\alpha (x_\beta)^2}{\gamma^2}\big)x^k x^l \\
        \notag
        &\quad+2x_\alpha x_\beta(\delta_\beta^k x^l +x^k \delta_\beta^l)+(x_\beta)^2(\delta_\alpha^k x^l+x^k \delta_\alpha^l)\bigg]\nu^\alpha\,d\sigma+C+O(\lambda^2\gamma^2)\\
        \label{eq:f1.0}
        &=:\int_{\partial B_\gamma}F_{ij}\partial^3_{\alpha\beta \beta}F_{ij}\nu^\alpha \,d\sigma+\frac{8}{3}\lambda^2\Xi_{i,j,\beta}+C+O(\lambda^2\gamma^2),
        \end{align}
        where $C=C(i,j,\beta, W^M)$.
Now, being $\nu^\alpha=x^\alpha\gamma^{-1}$ and $\abs{x}=\gamma$, summing over $\alpha $ we got
\begin{align*}
    \Xi_{i,j,\beta}&=\int_{\partial B_\gamma}W^M_{\mu ij\nu}W^Z_{kijl}\frac{x^\mu x^\nu}{\gamma^9}\Big[(2(x_\beta)^2+\gamma^2-16 (x_\beta)^2)x^k x^l+2\gamma^2x_\beta(\delta_\beta^k x^l+x^k \delta_\beta^l)+2(x_\beta)^2 x^kx^l\Big]\,d\sigma \\
    &=\int_{\partial B_\gamma}W^M_{\mu ij\nu}W^Z_{kijl}\frac{x^\mu x^\nu}{\gamma^9}\Big[-12(x_\beta)^2 x^k x^l+\gamma^2 x^k x^l+2\gamma^2 x_\beta(\delta_\beta^k x^l+x^k \delta_\beta^l)\Big]\,d\sigma.
\end{align*}
Summing now over $i,j,\beta$ and changing variables we find
\begin{align*}
    \sum_{i,j,\beta}\Xi_{i,j,\beta}=-\frac{4}{\gamma^7}\sum_{i,j}\int_{\partial B_\gamma}W^M_{\mu ij \nu}W^Z_{k i j l}x^\mu x^\nu x^k x^l\,d\sigma=-4\sum_{i,j}\int_{\Sp^3}W^M_{\mu ij \nu}W^Z_{kijl}z^\mu z^\nu z^k z^l\,d\sigma_{\Sp^3}(z).
\end{align*}
Using the identity (see e.g. \cite[Corollary 29]{brendle-2008-JAMS-Yamabe})
\begin{equation*}
    \int_{\Sp^3}z^\mu z^\nu z^k z^l\,d\sigma=\frac{\pi^2}{12}(\delta^{\mu\nu}\delta^{kl}+\delta^{\mu k}\delta^{\nu l}+\delta^{\mu l}\delta^{\nu k}),
\end{equation*}
one gets
\begin{align}
\notag
    \sum_{i,j,\beta}\Xi_{i,j,\beta}
    \notag
    &=-\frac{\pi^2}{3}\bigg[\sum_{i,j,\mu,k}W^M_{\mu i j \mu}W^Z_{kijk}+\sum_{i,j,k,l}\Big(W^M_{kijl}W^Z_{kijl}+W^M_{lijk}W^Z_{kijl}\Big)\bigg] \\
    \label{eq:f0.0}
    &=-\frac{\pi^2}{3}W^M\stell W^Z,
\end{align}
where in the last step we used the definition \eqref{eq:weyl-star-weyl} of $\stell$ and the trace-free property of $W$.
From this last equation and \eqref{eq:f1.0} we thus have
\begin{align}\label{eq:f1.1}
    \sum_{i,j,\beta}\int_{\partial B_\gamma}\dot{w}_{ij}\partial^3_{\alpha\beta\beta}\dot{w}_{ij}\nu^\alpha\,d\sigma=\sum_{i,j,\beta}\int_{\partial B_\gamma}F_{ij}\partial^3_{\alpha\beta\beta} F_{ij}\nu^\alpha\,d\sigma-\frac{8}{9}\pi^2\lambda^2 W^M\stell W^Z +C+O(\lambda^2\gamma^2),
\end{align}
where $C=C(W^M)$.

We now compute the terms in the second line of \eqref{eq:weyl-inner-exp}. Using \eqref{eq:gammaint-1-est}, we see that:
\begin{align}
\notag
    \sum_{i,j,\beta}\int_{\partial B_\gamma}\partial_\beta \dot{w}_{ij}\partial^{2}_{ij}\dot{w}_{\alpha\beta}\nu^\alpha\,d\sigma&=\sum_{i,j,\beta}\int_{\partial B_\gamma}\partial_\beta {F}_{ij}\partial^{2}_{ij}{F}_{\alpha\beta}\nu^\alpha\,d\sigma \\ 
    \notag
    &\quad+\frac{8}{3}\lambda^2\sum_{i,j,\beta}\int_{\Sp^3}\big(W^M_{\beta i j \nu}W^Z_{k\alpha \beta l}x^\nu +W^M_{\mu i j \beta} W^Z_{k \alpha \beta l}x^\mu\big)x^kx^l x_i x_j x^\alpha \,d\sigma \\
    \notag
    &\quad -\frac{32}{3}\lambda^2\sum_{i,j,\beta}\int_{\Sp^3}W^M_{\mu i j \nu}W^Z_{k \alpha \beta l}x^\mu x^\nu x_\beta x^k x^l x_i x_j x^\alpha\,d\sigma+C+O(\lambda^2\gamma^2) \\
    \label{eq:f1.2}
    &=\sum_{i,j,\beta}\int_{\partial B_\gamma}\partial_\beta {F}_{ij}\partial^{2}_{ij}{F}_{\alpha\beta}\nu^\alpha\,d\sigma +C+O(\lambda^2\gamma^2).
\end{align}
Here in the last step we used the symmetries of curvature type tensors, e.g., for $\beta,i$ fixed one has $\sum_{j,\nu}W^M_{\beta i j \nu}x_j x_\nu=W^M(\partial_\beta, \partial_i,x,x)=0$, where $x=(x_1,\dots,x_4)$.

Arguing in the same way, we also see that
\begin{equation}\label{eq:f1.3}
    \sum_{i,j,\beta}\int_{\partial B_\gamma}\partial_\beta\dot{w}_{ij}\partial^2_{\beta j}\dot{w}_{i\alpha}\nu^\alpha\,d\sigma= \sum_{i,j,\beta}\int_{\partial B_\gamma}\partial_\beta F_{ij}\partial^2_{\beta j}F_{i\alpha}\nu^\alpha\,d\sigma+C+O(\lambda^2\gamma^2).
\end{equation}
Using again \eqref{eq:gammaint-1-est} one has
\begin{align*}
    \sum_{i,j,\beta}\int_{\partial B_\gamma}\partial_\beta\dot{w}_{ij}\partial^2_{\alpha\beta}\dot{w}_{ij}\nu^\alpha\,d\sigma&=\sum_{i,j,\beta}\int_{\partial B_\gamma}\partial_\beta F_{ij}\partial^2_{\alpha\beta} F_{ij}\nu^\alpha\,d\sigma \\
    &-\frac{16}{3}\lambda^2\sum_{i,j}\int_{\Sp^3} W^M_{\mu i j \nu} W^Z_{kijl}x^\mu x^\nu x^k x^l \,d\sigma+C+O(\lambda^2\gamma^2),
\end{align*}
so that, after recalling \eqref{eq:f0.0}, we get
\begin{align}
\label{eq:f1.4}
    \sum_{i,j,\beta}\int_{\partial B_\gamma}\partial_\beta\dot{w}_{ij}\partial^2_{\alpha\beta}\dot{w}_{ij}\nu^\alpha\,d\sigma&=\sum_{i,j,\beta}\int_{\partial B_\gamma}\partial_\beta F_{ij}\partial^2_{\alpha\beta} F_{ij}\nu^\alpha\,d\sigma-\frac{4}{9}\pi^2\lambda^2 W^M \stell W^Z +C+O(\lambda^2\gamma^2).
\end{align}
Using the same method, one also sees that
\begin{align}
\label{eq:f1.5}
    \sum_{i,j,\beta}\int_{\partial B_\gamma}\partial_\alpha\dot{w}_{ij}\partial^2_{\beta\beta}\dot{w}_{ij}\nu^\alpha\,d\sigma&=\sum_{i,j,\beta}\int_{\partial B_\gamma}\partial_\alpha F_{ij}\partial^2_{\beta\beta} F_{ij}\nu^\alpha\,d\sigma-\frac{4}{9}\pi^2\lambda^2 W^M \stell W^Z +C+O(\lambda^2\gamma^2).
\end{align}

Finally, substituting \eqref{eq:f1.1}, \eqref{eq:f1.2}, \eqref{eq:f1.3}, \eqref{eq:f1.4} and \eqref{eq:f1.5} inside \eqref{eq:weyl-inner-exp}, the conclusion follows after recalling \eqref{eq:weyl-en-g_a-exp}.
\end{proof}

\subsection{Higher order terms in energy balance: outer boundary component}

It remains to compute the contribution to the expansion of $\w(w)$ coming from the integrals over $\partial B_1$ of \eqref{eq:weyl-interp-exp}. The procedure is the same adopted in the previous subsection.

Looking at the definition of $\dot{w}_{ij}$ \eqref{eq:w-wddot-relation} and using \eqref{eq:sph-harm-def}, \eqref{eq:sol-ii-form}, \eqref{eq:sol-ij-form}, \eqref{eq:rad-funct-eq}, \eqref{eq:raf-fun-tilde}, \eqref{eq:const-exp-1} and the values of $v_1, v_3$ coming from \eqref{eq:vec-ii-phi}, \eqref{eq:vec-ii-psi}, \eqref{eq:vec-ij-phi}, \eqref{eq:vec-ij-psi-hpi-st} we can write $\dot{w}_{ij}(x)$ as follows for $\abs{x}=1$:
\begin{align}
    \notag
    \dot{w}_{ij}&=\frac{1}{\abs{x}^2}\Big(-\frac{1}{3}W^M_{kijl}\frac{x^k x^l}{\abs{x}^2}\Big)+\abs{x}^2\Big(W^M_{kijl}\frac{x^k x^l}{\abs{x}^2}-\frac{1}{3}\lambda^2W^Z_{kijl}\frac{x^k x^l}{\abs{x}^2}\Big) \\
    \notag
    &\quad+\abs{x}^4\Big(-\frac{2}{3}W^M_{kijl}\frac{x^k x^l}{\abs{x}^2}\Big)+ \Upsilon_{\lambda,\gamma}(x) , \\
    \label{eq:wdot-1-exp}
    &=\Big(-\frac{1}{3\abs{x}^4}+1-\frac{2}{3}\abs{x}^2\Big)W^M_{kijl} x^k x^l-\frac{1}{3}\lambda^2 W^Z_{kijl} x^k x^l +\Upsilon_{\lambda,\gamma}(x),
\end{align}
where $\Upsilon_{\lambda,\gamma}(x)$ is an higher order error term such that $\abs{\nabla^k \Upsilon_{\lambda,\gamma}(x)}\leq C\gamma^4\lambda^2$ for $\abs{x}=1$ and $k=0,1,2,3$.

We now want to use \eqref{eq:wdot-1-exp} in order to compute 
\begin{align}
    \notag
    \Phi_1(w):&=-\frac{1}{2}\sum_{i,j,\beta}\int_{\partial B_1}\dot{w}_{ij} \partial^3_{\alpha\beta\beta} \dot{w}_{ij}\nu^\alpha\,d\sigma \\
    \label{eq:weyl-outer-exp}
    &+\sum_{i,j,\beta}\int_{\partial B_1}\Big[\big(\partial^2_{ij} \dot{w}_{\alpha \beta}-2\partial^2_{\beta j} \dot{w}_{i\alpha}+\partial^2_{\alpha\beta}\dot{w}_{ij}\big)\partial_\beta \dot{w}_{ij}-\frac{1}{2}\partial^2_{\beta \beta} \dot{w}_{ij} \partial_\alpha \dot{w}_{ij}\Big]\nu^\alpha \,d\sigma;
\end{align}
here we are always assuming the unit normal $\nu=(\nu^1,\dots,\nu^4)$ to be outward pointing.

First, we need an analogue of Lemma \ref{lemma:tecnical}:
\begin{lemma}\label{lem-technical-2}
    It holds
    \begin{align}
    \notag
        \int_{\partial B_1}\partial_\alpha \dot{w}_{ij}\partial^2_{\beta \tau}\dot{w}_{st}\nu^\eta\,d\sigma&=\lambda^4\int_{\partial B_1} \partial_\alpha H_{ij} \partial^2_{\beta \tau}H_{st}\nu^\eta \,d\sigma \\
        \label{eq:lem2-est}
        &+\frac{8}{3}\lambda^2\int_{\Sp^3}W^Z_{\mu ij \nu}W^M_{kstl}(\delta^\mu_\alpha x^\nu+x^\mu \delta_\alpha^\nu)x^k x^l x_\beta x_\tau x^\eta\,d\sigma+ O(\lambda^4\gamma^4),
    \end{align}
    where $H_{ij}$ is given by \eqref{eq:metr-err-def}.
\end{lemma}
\begin{proof}
    Differentiating the components of $\dot{w}_{ij}$ given by \eqref{eq:wdot-1-exp} and evaluating at $\abs{x}=1$ we see that:
    \begin{align}
    \notag
        \partial_\alpha\bigg[\Big(-\frac{1}{3\abs{x}^4}+1-\frac{2}{3}\abs{x}^2\Big)&W^M_{kijl} x^k x^l\bigg]=\Big(\frac{4}{3}\frac{x_\alpha}{\abs{x}^6}-\frac{4}{3}x_\alpha\Big)W^M_{kijl} x^k x^l\\
        \notag
        &+\Big(-\frac{1}{3\abs{x}^4}+1-\frac{2}{3}\abs{x}^2\Big)W^M_{kijl}(\delta_\alpha^k x^l +x_k\delta_{\alpha}^l),  
    \end{align}
    so in particular
    \begin{equation}\label{eq:der1-B_1}
        \Big(-\frac{1}{3\abs{x}^4}+1-\frac{2}{3}\abs{x}^2\Big)W^M_{kijl}x^kx^l\biggr\rvert_{\abs{x}=1}=0, \qquad \partial_\alpha\bigg[\Big(-\frac{1}{3\abs{x}^4}+1-\frac{2}{3}\abs{x}^2\Big)W^M_{kijl} x^k x^l\bigg]\biggr\rvert_{\abs{x}=1}=0.
    \end{equation}
    We then have
    \begin{align*}
        \partial^2_{\alpha\beta}\bigg[\Big(-\frac{1}{3\abs{x}^4}+&1-\frac{2}{3}\abs{x}^2\Big)W^M_{kijl} x^k x^l\bigg]=\Big(\frac{4}{3}\frac{\delta_{\alpha \beta}}{\abs{x}^6}-8\frac{x_\alpha x_\beta}{\abs{x}^8}-\frac{4}{3}\delta_{\alpha \beta}\Big)W^M_{kijl} x^k x^l \\
        &+\Big(\frac{4}{3}\frac{x_\alpha}{\abs{x}^6}-\frac{4}{3}x_\alpha\Big)W^M_{kijl}(\delta_\beta^k x^l+x^k \delta_\beta^l)+\Big(\frac{4}{3}\frac{x_\beta}{\abs{x}^6}-\frac{4}{3}x_\beta\Big)W^M_{kijl}(\delta_\alpha^k x^l+x^k \delta_\alpha^l) \\
        &+\Big(-\frac{1}{3\abs{x}^4}+1-\frac{2}{3}\abs{x}^2\Big)W^M_{kijl}(\delta_\alpha^k \delta_\beta^l +\delta_\beta^k \delta_\alpha^l), 
    \end{align*}
    from which one gets
    \begin{equation}\label{eq:der2-B_1}
        \partial^2_{\alpha\beta}\bigg[\Big(-\frac{1}{3\abs{x}^4}+1-\frac{2}{3}\abs{x}^2\Big)W^M_{kijl} x^k x^l\bigg]\biggr\rvert_{\abs{x}=1}=-8W^M_{kijl}x^kx^l x_\alpha x_\beta,
    \end{equation}
    and 
    \begin{align}
        \notag
        \partial^3_{\alpha\beta\beta}\bigg[\Big(-\frac{1}{3\abs{x}^4}+&1-\frac{2}{3}\abs{x}^2\Big)W^M_{kijl} x^k x^l\bigg]\biggr\rvert_{\abs{x}=1}=\big(-16\delta_{\alpha \beta}x_\beta-8 x_\alpha +64 x_\alpha (x_\beta)^2\big)W^M_{kijl}x^k x^l \\
        \label{eq:der3-B_1}
        &-16x_\alpha x_\beta W^M_{kijl}(\delta_\beta^k x^l+ x^k \delta_\beta^l)-8 (x_\beta)^2W^M_{kijl}(\delta_\alpha^k x^l +x^k \delta_\alpha^l).
    \end{align}
    Finally, using \eqref{eq:wdot-1-exp}, \eqref{eq:der1-B_1} and \eqref{eq:der2-B_1}, one obtains \eqref{eq:lem2-est}.
\end{proof}

We can now compute the expansion of \eqref{eq:weyl-outer-exp}:
\begin{lemma}\label{lem:est-1}
    Let $\Phi_1(w)$ be as in \eqref{eq:weyl-outer-exp}; then
    \begin{equation*}
        \Phi_1(w)={a^{-4}}\int_{B_1}\abs{W^{g_b}}^2\,dV_{g_b}-\frac{2}{9}\pi^2\lambda^2 W^M(p)\stell W^Z(q)+O(\lambda^4\gamma^4)+O(b^2),
    \end{equation*}
    where $W^M(p)\stell W^Z(q)$ is defined in \eqref{eq:weyl-star-weyl}.
\end{lemma}
\begin{proof}
    Using \eqref{eq:wdot-1-exp} and \eqref{eq:der3-B_1}, we can compute
    \begin{align*}
        \int_{\partial B_1}\dot{w}_{ij}&\partial^3_{\alpha\beta \beta}\dot{w}_{ij}\nu^\alpha\,d\sigma=\lambda^4\int_{\partial B_1} H_{ij}\partial^3_{\alpha\beta \beta} H_{ij}\nu^\alpha\,d\sigma+\lambda^2\int_{\partial B_1}H_{ij}\bigg[-8 (x_\beta)^2W^M_{kijl}(\delta_\alpha^k x^l +x^k \delta_\alpha^l) \\
        \big(&-16\delta_{\alpha \beta}x_\beta-8 x_\alpha +64 x_\alpha (x_\beta)^2\big)W^M_{kijl}x^k x^l-16x_\alpha x_\beta W^M_{kijl}(\delta_\beta^k x^l+ x^k \delta_\beta^l)\bigg]x^\alpha\,d\sigma+O(\lambda^4\gamma^4) \\
        &=\lambda^4\int_{\partial B_1} H_{ij}\partial^3_{\alpha\beta \beta} H_{ij}\nu^\alpha\,d\sigma+\frac{8}{3}\lambda^2\int_{\Sp^3}W^Z_{\mu i j \nu}W^M_{kijl}x^\mu x^\nu\bigg[2(x_\beta^2)x^kx^l \\
        & \qquad \qquad \qquad\quad  +\big( 2(x_\beta)^2+1-8(x_\beta)^2\big)x^k x^l+2 x_\beta(\delta_\beta^k x^l +x^k \delta_\beta^l)\bigg]\,d\sigma+O(\lambda^4\gamma^4) \\
        &=:\lambda^4\int_{\partial B_1} H_{ij}\partial^3_{\alpha\beta \beta} H_{ij}\nu^\alpha\,d\sigma+\frac{8}{3}\lambda^2 \Theta_{i,j,\beta}+O(\lambda^4\gamma^4).
    \end{align*}
    Summing over $i,j,\beta$, one has
    \begin{align*}
\sum_{i,j,\beta}\Theta_{i,j,\beta}=4\sum_{i,j}\int_{\Sp^3}W^Z_{\mu i j \nu}W^M_{kijl}x^\mu x^\nu x^k x^l\,d\sigma=\frac{\pi^2}{3} W^M \stell W^Z, 
    \end{align*}
    where the last equality follows from \eqref{eq:f0.0}. Hence
\begin{align}\label{eq:f2.1}
    \sum_{i,j,\beta} \int_{\partial B_1}\dot{w}_{ij}\partial^3_{\alpha\beta \beta}\dot{w}_{ij}\nu^\alpha\,d\sigma=\lambda^4\sum_{i,j,\beta}\int_{\partial B_1} H_{ij}\partial^3_{\alpha\beta \beta} H_{ij}\nu^\alpha\,d\sigma+\frac{8}{9}\pi^2\lambda^2 W^M \stell W^Z +O(\lambda^4\gamma^4).
\end{align}
Now, by applying Lemma \ref{lem-technical-2} we see that
\begin{align}
    \notag
    \sum_{i,j,\beta}\int_{\partial B_1}\partial_\beta \dot{w}_{ij}\partial^{2}_{ij}\dot{w}_{\alpha\beta}\nu^\alpha\,d\sigma&=\lambda^4\sum_{i,j,\beta}\int_{\partial B_1}\partial_\beta {H}_{ij}\partial^{2}_{ij}{H}_{\alpha\beta}\nu^\alpha\,d\sigma \\ 
    \notag
    +\frac{8}{3}\lambda^2&\sum_{i,j,\beta}\int_{\partial B_1}\big(W^Z_{\beta i j \nu}W^M_{k\alpha \beta l}x^\nu +W^Z_{\mu i j \beta} W^M_{k \alpha \beta l}x^\mu\big)x^kx^l x_i x_j x^\alpha \,d\sigma+O(\lambda^4\gamma^4) \\
    \label{eq:f2.2}
    &=\lambda^4\sum_{i,j,\beta}\int_{\partial B_1}\partial_\beta {F}_{ij}\partial^{2}_{ij}{F}_{\alpha\beta}\nu^\alpha\,d\sigma +O(\lambda^4\gamma^4),
\end{align} 
    where in the last step we used the skew-symmetry of $W$. Similarly,
    \begin{equation}\label{eq:f2.3}
        \sum_{i,j,\beta}\int_{\partial B_1}\partial_\beta\dot{w}_{ij}\partial^2_{\beta j}\dot{w}_{i\alpha}\nu^\alpha\,d\sigma=\lambda^4\sum_{i,j,\beta}\int_{\partial B_1}\partial_\beta H_{ij}\partial^2_{\beta j} H_{i\alpha}\nu^\alpha\,d\sigma+O(\lambda^4\gamma^4).
    \end{equation}
    Moreover, by Lemma \ref{lem-technical-2} and \eqref{eq:f0.0} we also have
    \begin{align}
    \notag
    \sum_{i,j,\beta}\int_{\partial B_1}\partial_\beta\dot{w}_{ij}\partial^2_{\alpha\beta}\dot{w}_{ij}\nu^\alpha\,d\sigma&=\lambda^4\sum_{i,j,\beta}\int_{\partial B_1}\partial_\beta H_{ij}\partial^2_{\alpha\beta} H_{ij}\nu^\alpha\,d\sigma \\
    \notag
    &\quad+\frac{16}{3}\lambda^2\sum_{i,j}\int_{\Sp^3} W^Z_{\mu i j \nu} W^M_{kijl}x^\mu x^\nu x^k x^l \,d\sigma+O(\lambda^4\gamma^4) \\
    \label{eq:f2.4}
    &=\lambda^4\sum_{i,j,\beta}\int_{\partial B_1}\partial_\beta H_{ij}\partial^2_{\alpha\beta} H_{ij}\nu^\alpha\,d\sigma+\frac{4}{9}\pi^2\lambda^2 W^M \stell W^Z +O(\lambda^4\gamma^4),
\end{align}
and
\begin{align}
\label{eq:f2.5}
    \sum_{i,j,\beta}\int_{\partial B_1}\partial_\alpha\dot{w}_{ij}\partial^2_{\beta\beta}\dot{w}_{ij}\nu^\alpha\,d\sigma&=\lambda^4\sum_{i,j,\beta}\int_{\partial B_1}\partial_\alpha H_{ij}\partial^2_{\beta\beta} H_{ij}\nu^\alpha\,d\sigma+\frac{4}{9}\pi^2\lambda^2 W^M \stell W^Z +O(\lambda^4\gamma^4).
\end{align}
Finally, substituting \eqref{eq:f2.1}, \eqref{eq:f2.2}, \eqref{eq:f2.3}, \eqref{eq:f2.4} and \eqref{eq:f2.5} inside \eqref{eq:weyl-outer-exp} and recalling \eqref{eq:weyl-en-g_b-exp} we conclude our proof.
\end{proof}

\subsection{Computation of the energy balance}

We now want to prove the following result:
\begin{proposition}\label{prop:energy-balance}
    Assume $0<a\ll \gamma \ll \lambda^{-1}<1$. Then
    \begin{equation}\label{eq:en-bal-fin}
        \w^X(g_X)-\w^M(g_M)-\w^Z(g_Z)=a^4\Big(-\frac{4}{9}\pi^2 \lambda^2 W^M(p)\stell W^Z(q) +C +O(\lambda^2\gamma^2)\Big) +O\big(a^\frac{9}{2}\gamma^{-5}\big),
    \end{equation}
    where $C$ is a constant depending on $W^M(p)$.
\end{proposition}

\begin{proof}
By definition of $g_X$ (cf. \eqref{eq:g_X-def}) we have $g_X\equiv w$ in $B_1\backslash B_\gamma$. Moreover, it holds
\begin{align*}
    \w(w)=\int_{B_1\backslash B_\gamma} \abs{W^{g_X}}^2\,dV_{g_X}=a^4\big( \Phi_\gamma(w)+\Phi_1(w)\big)+ O(a^6),
\end{align*}
see \eqref{eq:weyl-interp-exp}, \eqref{eq:weyl-inner-exp}, \eqref{eq:weyl-outer-exp}. As a consequence, by virtue of \eqref{eq:weyl-en-g_b-exp}, \eqref{eq:weyl-en-g_a-exp}, Lemma \ref{lem:est-gamma} and Lemma \ref{lem:est-1}, formula \eqref{eq:energ-balance-formula} turns into
\begin{align}
\notag
    \w^X(g_X)-\w^M(g_M)-\w^Z(g_Z)&=a^4\Big(-\frac{4}{9}\pi^2\lambda^2 W^M(p)\stell W^Z(q)+C+O(\lambda^2\gamma^2)\Big)+O(a^6) \\
    \label{eq:err-main}
    +\int_{(B_{1+\sqrt{a}}\backslash B_1)\cup(B_\gamma\backslash B_{\gamma-\sqrt{a}})}\abs{W^{g_X}}&^2\,dV_{g_X}-\int_{B_{1+\sqrt{a}}\backslash B_1}\abs{W^{g_b}}^2\,d V_{g_b}-\int_{B_{\gamma}\backslash B_{\gamma-\sqrt{a}}}\abs{W^{g_a}}^2\,dV_{g_a}.
\end{align}
The Proposition now follows from the next Lemma.
\end{proof}

\begin{lemma}\label{lem:error-integrals}
    Assume $0<a\ll \gamma \ll \lambda^{-1}<1$. Then
    \begin{align*}
        \int_{(B_{1+\sqrt{a}}\backslash B_1)\cup(B_\gamma\backslash B_{\gamma-\sqrt{a}})}&\abs{W^{g_X}}^2\,dV_{g_X}+\int_{B_{1+\sqrt{a}}\backslash B_1}\abs{W^{g_b}}^2\,d V_{g_b}+\int_{B_{\gamma}\backslash B_{\gamma-\sqrt{a}}}\abs{W^{g_a}}^2\,dV_{g_a}\leq C\frac{a^{\frac{9}{2}}}{\gamma^5},
    \end{align*}
    for a suitable constant $C>0$ which \emph{does not} depend upon $a,\gamma,\lambda$.
\end{lemma}
\begin{proof}
     To begin recall that, in local coordinates, one has
     \begin{equation*}
         R_{ijkl}=(\partial_i \Gamma_{jk}^s  -\partial_j\Gamma_{ik}^s) g_{sl}+(\Gamma_{jk}^s\Gamma_{is}^t-\Gamma_{ik}^s\Gamma_{js}^t)g_{tl},
     \end{equation*}
     where 
     \begin{gather*}
     \Gamma_{ij}^k=\frac{1}{2}g^{ks}(\partial_i g_{js}+\partial_j g_{is}-\partial_s g_{ij}), \\
         \partial_l \Gamma_{ij}^k=\frac{1}{2}(\partial_lg^{ks})(\partial_i g_{js}+\partial_j g_{is}-\partial_s g_{ij})+\frac{1}{2}g^{ks}(\partial^2_{il}g_{js}+\partial^2_{jl}g_{is}-\partial^2_{sl}g_{ij}).
     \end{gather*}
As a consequence, one has $\abs{R_{ijkl}}\leq C\big(\lvert\nabla g^{-1}\rvert\abs{\nabla g}+ \abs{\nabla g}^2 +\lvert\nabla^2 g\rvert\big)$, where the gradients and norms are taken with respect to the Euclidean metric, and $C=C(\vert g\vert$, $\vert g^{-1} \vert$). Moreover, by writing the Weyl tensor in local coordinates (see \eqref{eq:weyl-local-coord-expr}), one also has the same type of estimate for $W^g$. Hence
\begin{equation}\label{eq:weyl-est-rough}
        \abs{W^g}_g^2=W^{ijkl}W_{ijkl}\leq C\big(\lvert\nabla g^{-1}\rvert^2\abs{\nabla g}^2+ \abs{\nabla g}^4 +\lvert\nabla^2 g\rvert^2\big),
    \end{equation}
always with a suitable $C=C(\vert g\vert$, $\vert g^{-1} \vert$).
Taking now as $\{x^i\}$ the (conformal) normal coordinates at $q$ for the scaled metric $b^{-2}g_Z$, then the metrics $g_b$ and $g_a$ expand as in \eqref{eq:g_b-exp} and \eqref{eq:g_a-exp-fin} respectively, so that \eqref{eq:weyl-est-rough} holds (for both $g_a$ and $g_b$) in the region $\{ \gamma/2\leq\abs{x}\leq 2\}$ with a constant $C$ which does not depend on $a,b\ll1$ (we need $a,b$ small in order for the expansions to hold in such region).

Now, by \eqref{eq:g_a-exp-fin}, \eqref{eq:err-g_a-fin} one has that
    \begin{equation}\label{eq:g_a-der-est}
        \abs{\nabla g_a}\leq C\frac{a^2}{\abs{x}^3}, \qquad  \qquad\abs{\nabla g_a^{-1}}\leq C\frac{a^2}{\abs{x}^3},\qquad \abs{\nabla^2 g_a}\leq C\frac{a^2}{\abs{x}^4}, \qquad \text{for $\abs{x}\geq \frac{\gamma}{2}$},
    \end{equation}
    (the second estimate actually follows by writing $g_a^{-1}$ as a Neumann series and noticing that it has an expansion similar to that of $g_a$). Applying \eqref{eq:g_a-der-est} together with \eqref{eq:weyl-est-rough}, one has
    \begin{align}
    \notag
        \int_{B_{\gamma}\backslash B_{\gamma-\sqrt{a}}}\abs{W^{g_a}}^2\,dV_{g_a}&\leq C\int_{B_{\gamma}\backslash B_{\gamma-\sqrt{a}}}\Big(\frac{a^8}{\abs{x}^{12}}+\frac{a^4}{\abs{x}^8}\Big)\,dx\leq Ca^4\int_{B_{\gamma}\backslash B_{\gamma-\sqrt{a}}}\abs{x}^{-8}\,dx \\
        \label{eq:fff}
        &=-\frac{\tilde{C}}{4}a^4\Big(\frac{1}{\gamma^4}-\frac{1}{\gamma^4(1-\frac{\sqrt{a}}{\gamma})^4}\Big)=-\frac{\tilde{C}}{4}a^4\Big(-4\frac{\sqrt{a}}{\gamma^5}+O\Big(\frac{a}{\gamma^6}\Big)\Big)\leq C\frac{a^{\frac{9}{2}}}{\gamma^5},
    \end{align}
where we used the fact that $a\ll\gamma$.

Similarly, by \eqref{eq:g_b-exp}, \eqref{eq:err-g_b} one has
\begin{equation*}
        \abs{\nabla g_b}\leq C b^2 \abs{x}, \qquad  \qquad\abs{\nabla g_b^{-1}}\leq Cb^2{\abs{x}},\qquad \abs{\nabla^2 g_b}\leq Cb^2, \qquad \text{for $\abs{x}\leq 2$},
    \end{equation*}
    and coupling those estimates with \eqref{eq:weyl-est-rough} one gets
    \begin{align*}
        \int_{B_{1+\sqrt{a}}\backslash B_1}\abs{W^{g_b}}^2\,d V_{g_b}\leq C\int_{B_{1+\sqrt{a}}\backslash B_1}(b^8 \abs{x}^4 +b^4)\,dx\leq C b^4\Big((1+\sqrt{a})^4-1\Big)\leq C b^4 \sqrt{a}=Ca^{\frac{9}{2}}\lambda^4.
    \end{align*}

We now turn our attention to the error term of $g_X$; by \eqref{eq:glued-metr-ext} we know that, in $\gamma-\sqrt{a}\leq\abs{x}\leq\gamma$ one has
\begin{align*}
    \nabla ((g_X)_{ij})(x)&=-\frac{1}{3}a^2 W^M_{kijl}\nabla\Big(\frac{x^k x^l}{\abs{x}^4}\Big)-\chi_a^{'}(\gamma-\abs{x})\frac{x}{\abs{x}}\zeta^a_{ij}(x)+\chi_a(\gamma-\abs{x})\nabla \zeta^a_{ij}(x),  \\
    \nabla^2((g_X)_{ij})(x)&=-\frac{1}{3}a^2W^M_{kijl}\nabla^2\Big(\frac{x^k x^l}{\abs{x}^4}\Big)+\Big(\chi_a^{''}(\gamma-\abs{x})\frac{x\otimes x}{\abs{x}^2}-\chi_a^{'}(\gamma-\abs{x})\nabla^2\big(\abs{x}\big)\Big)\zeta^a_{ij}(x) \\
    &-\chi_a^{'}(\gamma-\abs{x})    \Big(\frac{x}{\abs{x}}\otimes\nabla \zeta_{ij}^a(x)+\nabla\zeta^a_{ij}(x)\otimes\frac{x}{\abs{x}}\Big)+\chi_a(\gamma-\abs{x})\nabla^2 \zeta^a_{ij}(x).
\end{align*}
Moreover, writing $g_X^{-1}$ as a Neumann series one finds that, for $\gamma-\sqrt{a}\leq\abs{x}\leq\gamma$,
\begin{equation*}
    (g_X^{-1})_{ij}(x)=\delta_{ij}+\frac{1}{3}a^2 W^M_{kijl}\frac{x^k x^l}{\abs{x}^4}-\zeta^a_{ij}(x)\chi_a(\gamma-\abs{x})+\theta^a_{ij}(x),
\end{equation*}
where
\begin{equation*}
    \abs{{\theta}^a_{ij}(x)}\abs{x}^{3}+\abs{\nabla {\theta}^a_{ij}(x)} \abs{x}^{4}+\abs{\nabla^2 {\theta}^a_{ij}(x)}\abs{x}^{5}\leq C a^{3}, 
\end{equation*}
(recall that $a\ll \gamma$). Using these formulae together with \eqref{eq:err-g_a-fin}, \eqref{eq:cutoff-eta-def}, one finds, for $\gamma-\sqrt{a}\leq\abs{x}\leq\gamma$,
\begin{equation*}
     \abs{\nabla g_X}\leq C\frac{a^2}{\abs{x}^3}, \qquad  \qquad\abs{\nabla g_X^{-1}}\leq C\frac{a^2}{\abs{x}^3},\qquad \qquad \abs{\nabla^2 g_X}\leq C\frac{a^2}{\abs{x}^4}.
\end{equation*}
We can now argue as for \eqref{eq:fff} to conclude that
\begin{equation*}
     \int_{B_{\gamma}\backslash B_{\gamma-\sqrt{a}}}\abs{W^{g_X}}^2\,dV_{g_X}\leq C\frac{a^{\frac{9}{2}}}{\gamma^5}.
\end{equation*}
Finally, a similar argument in the region $1\leq\abs{x}\leq1+\sqrt{a}$ also shows that
 \begin{align*}
        \int_{B_{1+\sqrt{a}}\backslash B_1}\abs{W^{g_X}}^2\,d V_{g_X}\leq C b^4 \sqrt{a}=Ca^{\frac{9}{2}}\lambda^4.
    \end{align*} 
    This concludes our proof.
\end{proof}

\section{Proof of Theorem \ref{thm:main}}\label{sec:proofs of main thms}

We now want to employ Proposition \ref{prop:energy-balance} to prove Theorem \ref{thm:main}. This amounts to study the sign of \eqref{eq:en-bal-fin} and hence the sign of the interaction term \eqref{eq:weyl-star-weyl}, which we recall is
\begin{equation}\label{eq:weyl-star-weyl-2}
    W^M(p)\stell W^Z(q) =\sum_{i,j,k,l}W^Z_{kijl}(q)\big(W^M_{kijl}(p)+W^M_{lijk}(p)\big).
\end{equation}
To begin, we show that, by using our freedom to relatively rotate the normal coordinate systems, we can rewrite \eqref{eq:weyl-star-weyl-2} as a multiple of the {\em scalar product} between $W^M(p)$ and $W^Z(q)$:

\begin{lemma}\label{lem:wstarw.expre-scal}
    For a suitable choice of (conformal) normal coordinates at $p$ and $q$, one has
    \begin{equation}\label{eq:gv-lemma-formula}
        W^M(p) \stell W^Z(q)=\frac{3}{2}\sum_{i,j,k,l}W^M_{kijl}(p)W^Z_{kijl}(q).
    \end{equation}
\end{lemma}
\begin{proof}
    $W^M(p)$ can be regarded as a trace-free symmetric endomorphism $\widehat{W}^M(p)$ on $\Lambda^2(T^*_pM)$ which preserves the spaces $\Lambda^2_{\pm}(T^*_pM)$ (see Section \ref{sec:prelims}). As a consequence, we can diagonalize $\widehat{W}^M(p)$ by choosing an orthogonal basis of eigenforms $\{\omega^+,\eta^+,\theta^+,\omega^-,\eta^-,\theta^-\}$, such that $\{\omega^+,\eta^+,\theta^+\}$ is a basis for $\Lambda^2_+(T^*_pM)$ and $\{\omega^-,\eta^-,\theta^-\}$ is a basis for $\Lambda^2_-(T^*_pM)$.
    Even more, by \cite[Lemma 2]{derdzinski-1983-Compositio} we know that, given \emph{any} oriented orthogonal basis $\{\omega^+,\eta^+,\theta^+\}$ for $\Lambda^2_+(T^*_pM)$ and \emph{any} oriented orthogonal basis $\{\omega^-,\eta^-,\theta^-\}$ for
    $\Lambda^2_-(T^*_pM)$ whose vectors have all length $\sqrt{2}$, then there exists an oriented orthonormal basis $\{e_1,\dots,e_4\}$ for $T_pM$ such that
    \begin{align}
    \notag
        \omega^+&=e^1\wedge e^2+ e^3\wedge e^4 &  \omega^-&= e^1\wedge e^2- e^3\wedge e^4 \\
        \label{eq:basis-forms}
        \eta^+&=e^1\wedge e^3+ e^4\wedge e^2 &  \eta^-&= e^1\wedge e^3- e^4\wedge e^2 \\
        \notag
        \theta^+&=e^1\wedge e^4+e^2\wedge e^3 & \theta^-&=e^1\wedge e^4- e^2\wedge e^3.
    \end{align}
    Here $\{e^1,\dots e^4\}$ is the dual basis and the orientation on $\Lambda^2_{\pm}$ is the one induced by $M$. 
    If we now choose normal coordinates in $p$ such that $\partial_i\mid_p=e_i \,\, \forall i=1,\dots,4$, then $\widehat{W}_{\pm}^M(p)$ are diagonalized with respect to the associated bases \eqref{eq:basis-forms} for $\Lambda^2_{\pm}(T_p^*M)$.

Clearly, we can do the exact same thing for $W^Z(q)$.

As a consequence, when we identify the orthonormal bases for $T_pM$ and $T_qZ$ during the gluing procedure (i.e. we implicitly have $\partial_i^M\vert_p=e_i=\partial_i^Z\vert_q$ from \eqref{eq:identif-map}), we then find out that $\widehat{W}_{\pm}^M(p)$ and $\widehat{W}_{\pm}^Z(q)$ are \emph{simultaneously diagonalized} by the bases \eqref{eq:basis-forms}. We are now in position to apply \cite[Lemma 12.2]{gursky-viaclovsky-2016-Advances} (the assumption about simultaneous diagonalization is stated at the beginning of Section 12.1), which readily implies \eqref{eq:gv-lemma-formula}.  
\end{proof}
It is now possible to study the sign of \eqref{eq:weyl-star-weyl-2}:
\begin{lemma}\label{lem:wstarw-positivity}
    Consider normal coordinates at $p\in M$ and $q\in Z$ as in the proof of Lemma \ref{lem:wstarw.expre-scal}. Then, if $M$ and $Z$ are \emph{not} locally conformally flat and \emph{not} one self-dual and the other anti-self-dual, one has 
    \begin{equation}\label{eq:weylstarw-sign}
    W^M(p) \stell W^Z(q)>0    
    \end{equation}
    for a suitable choice of the basepoints $p\in M$ and $q\in Z$.
\end{lemma}

\begin{proof}
 In the orthogonal eigenbases for $T_pM$ and $T_qZ$ \eqref{eq:basis-forms} given by the previous lemma, \\ $\widehat{W}^M:=\widehat{W}^M(p)$ and $\widehat{W}^Z:=\widehat{W}^Z(q)$ write down as follows:
    \begin{align*}
        \widehat{W}^M=\frac{1}{2}\big(\lambda^+_M \omega^+\otimes \omega^+_\flat +  &\mu^+_M \eta^+\otimes \eta^+_\flat +\nu^+_M \theta^+\otimes \theta^+_\flat\big) \\
        &+\frac{1}{2}\big(\lambda^-_M \omega^-\otimes \omega^-_\flat + \mu^-_M \eta^-\otimes \eta^-_\flat +\nu^-_M \theta^-\otimes \theta^-_\flat\big), \\
        \widehat{W}^Z=\frac{1}{2}\big(\lambda^+_Z \omega^+\otimes \omega^+_\flat +  &\mu^+_Z \eta^+\otimes \eta^+_\flat +\nu^+_Z \theta^+\otimes \theta^+_\flat\big) \\
        &+\frac{1}{2}\big(\lambda^-_Z \omega^-\otimes \omega^-_\flat + \mu^-_Z \eta^-\otimes \eta^-_\flat +\nu^-_Z \theta^-\otimes \theta^-_\flat\big),
    \end{align*}
    where $\lambda^\pm_M,\mu^\pm_M,\nu^\pm_M$ and $\lambda^\pm_Z,\mu^\pm_Z,\nu^\pm_Z$ denote the corresponding eigenvalues and, e.g., $\omega_\flat$ is the $2$-vector corresponding to $\omega$ through the scalar product on the tangent (that is, the Euclidean one). In virtue of the trace-free property of $W_{\pm}^M$ and $W^Z_\pm$, one has
    \begin{equation}\label{eq:tfprop}
    \lambda^\pm_M+\mu^\pm_M+\nu^\pm_M=0 \qquad \text{and} \qquad \lambda^\pm_Z+\mu^\pm_Z+\nu^\pm_Z=0;
    \end{equation}
    moreover,
    \begin{gather}
    \notag
        (\lambda^\pm_M)^2+(\mu^\pm_M)^2+(\nu^\pm_M)^2=\big\lVert\widehat{W}^M_\pm\big\rVert^2=\frac{1}{4}\abs{W^M_\pm}^2 \\
        \label{eq:weylop-norm}
        (\lambda^\pm_Z)^2+(\mu^\pm_Z)^2+(\nu^\pm_Z)^2=\big\lVert\widehat{W}^Z_\pm\big\rVert^2=\frac{1}{4}\abs{W^Z_\pm}^2,
    \end{gather}
    where $\norm{\cdot}$ denotes the induced operator norm on $2$-forms. Without loss of generality, assume also that the eigenvalues satisfy
    \begin{align}\label{eq:eigen-monot}
        \lambda^\pm_M\geq \mu^\pm_M\geq \nu^\pm_M, \qquad \lambda^\pm_Z\geq \mu^\pm_Z\geq \nu^\pm_Z;
    \end{align}
    this is due to the fact that we can generate all possible permutations of each triple of eigenvalues by a orientation-preserving change of basis (we first permute the basis of $2$-forms and then apply \cite[Lemma 2]{derdzinski-1983-Compositio} as in the previous Lemma in order to get the basis for the tangent).
    
    By \eqref{eq:gv-lemma-formula},
    \begin{align}
    \notag
        W^M(p) \stell W^Z(q)&=\frac{3}{2}\sum_{i,j,k,l}W^M_{kijl}(p)W^Z_{kijl}(q)=\frac{3}{2}\langle W^M,W^Z \rangle= 6 \langle \widehat{W}^M,\widehat{W}^Z\rangle \\
        \label{eq:eqrt}
        &=6\big( \lambda^+_M\lambda^+_Z+\mu^+_M\mu^+_Z+\nu^+_M\nu^+_Z +\lambda^-_M\lambda^-_Z+\mu^-_M\mu^-_Z+\nu^-_M\nu^-_Z\big).
    \end{align}
     Consider then
    \begin{equation*}
        \lambda^\pm_M\lambda^\pm_Z+\mu^\pm_M\mu^\pm_Z+\nu^\pm_M\nu^\pm_Z=(\lambda^\pm_M,\mu^\pm_M,\nu^\pm_M)\cdot(\lambda^\pm_Z,\mu^\pm_Z,\nu^\pm_Z)=:v_1\cdot v_2=\abs{v_1}\abs{v_2}\cos(\theta),
    \end{equation*}
    where $\theta$ is the angle between the two vectors. By \eqref{eq:tfprop}, $v_1$ and $v_2$ belong to the plane $\pi=\{x_1+x_2+x_3=0\}$, while, by virtue of \eqref{eq:eigen-monot}, we further have that $v_1$ and $v_2$ lie inside the cone sector of $\pi$ whose edges are given by the lines of directions $\frac{1}{\sqrt{6}}(1,1,-2)$ and $\frac{1}{\sqrt{6}}(2,-1,-1)$ respectively. Thus the opening angle of the cone is equal to $\arccos(\frac{1}{2})=\frac{\pi}{3}$. As a consequence of this and \eqref{eq:weylop-norm} one has
    \begin{equation*}
         \lambda^\pm_M\lambda^\pm_Z+\mu^\pm_M\mu^\pm_Z+\nu^\pm_M\nu^\pm_Z\geq \frac{1}{2}\abs{(\lambda^\pm_M,\mu^\pm_M,\nu^\pm_M)}\abs{(\lambda^\pm_Z,\mu^\pm_Z,\nu^\pm_Z)}=\frac{1}{2}\big\lVert\widehat{W}^M_\pm\big\rVert \big\lVert\widehat{W}^Z_\pm\big\rVert.
    \end{equation*}
    Now, inasmuch as $M$ and $Z$ are not locally conformally flat and not one self-dual and the other anti-self-dual, by the previous equation, \eqref{eq:eqrt} and a suitable choice of $p$ and $q$, we immediately get \eqref{eq:weylstarw-sign}. 
\end{proof}

\medskip

\begin{proof}[Proof of Theorem \ref{thm:main}]
    As already pointed out in Remark \ref{rem:manif-orientation}, the manifold $X$ defined in Section \ref{sec:gluing-setup} is diffeomorphic to $Z\#\overbar{M}$. Conversely, if we perform the same gluing procedure with $\overbar{M}$ in place of $M$, we end up with a manifold $Y$ satisfying $Y\overset{\mathrm{diff}}{\cong}Z\#M$. Applying now Proposition \ref{prop:energy-balance} in this setting, we have
    \begin{align}
    \notag
        \w^Y(g_Y)-\w^{{M}}(g_{{M}})-\w^Z(g_Z)&=\w^Y(g_Y)-\w^{\overbar{M}}(g_{{M}})-\w^Z(g_Z)  \\
        \label{eq:finf}
        &=a^4\underbrace{\Big(C-\frac{4}{9}\pi^2 \lambda^2 W^{\overbar{M}}(p)\stell W^Z(q)  +O(\lambda^2\gamma^2)\Big)}_{:=E_{\lambda,\gamma}} +O\big(a^\frac{9}{2}\gamma^{-5}\big),
    \end{align}
    where $C=C(W^{\overbar{M}}(p))$.
    Being $W_{\pm}^M=W^{\overbar{M}}_{\mp}$, if $M$ and $Z$ are not both self-dual or both anti-self-dual, then $\overbar{M}$ and $Z$ are not one self-dual and the other anti-self-dual; hence we can apply Lemma \ref{lem:wstarw-positivity} to get
    \begin{equation*}
        W^{\overbar{M}}(p)\stell W^Z(q)>0.
    \end{equation*}
    We can now fix $\lambda>0$ large enough to have e.g.
\begin{equation*}
    C-\frac{4}{9}\pi^2 \lambda^2 W^{\overbar{M}}(p)\stell W^Z(q) <-2.
\end{equation*}
It is then possible to fix $0<\gamma\ll\lambda$ such that (\eqref{eq:const-exp-1} holds and) $E_{\lambda,\gamma}<-1$. After that, we fix $0<a< \gamma$ such that the RHS of \eqref{eq:finf} is negative. This shows that $g_Y$ satisfies \eqref{eq:en-inequality}. 

Finally, by a density argument we also get a smooth metric $\tilde{g}_Y$ still satisfying \eqref{eq:en-inequality}.
\end{proof}

\appendix
\section{Appendix}

\begin{proof}[Proof of \eqref{eq:metr-inv-exp-cnc}]
    Letting $I(z)=\frac{z}{\abs{z}^2}=:y$ denote the inversion map and $y=(y^1,\dots,y^4)$ the inverted conformal normal coordinates, from \eqref{eq:metr-exp-cnc} and the definition of $F$ we got
    \begin{align*}
        g_N&=(F\circ I)^2 I^*\Big(\big[\delta_{\alpha\beta}-\frac{1}{3}W_{k\alpha\beta l}z^kz^l+O^{(3)}\big(\abs{z}^3\big)\big]dz^\alpha dz^\beta\Big) \\
        &=\big(\abs{y}^2\big)^2\Big(\delta_{\alpha\beta}-\frac{1}{3}W_{k\alpha\beta l}\frac{y^ky^l}{\abs{y}^4}+O^{(3)}\big(\abs{y}^{-3}\big)\Big)\frac{1}{\abs{y}^2}\Big(\delta^{\alpha}_s-2\frac{y^{\alpha} y_s}{\abs{y}^2}\Big)dy^s\frac{1}{\abs{y}^2}\Big(\delta^{\beta}_t-2\frac{y^{\beta} y_t}{\abs{y}^2}\Big)dy^t,
    \end{align*}
where we used 
\begin{equation*}
    I^*(dz^{\alpha})=d\Big(\frac{y^{\alpha}}{\abs{y}^2}\Big)=\frac{dy^{\alpha}}{\abs{y}^2}-2\frac{y^{\alpha}(y\cdot dy)}{\abs{y}^4}=\frac{1}{\abs{y}^2}\Big(\delta^{\alpha }_s-2\frac{y^{\alpha} y_s}{\abs{y}^2}\Big) dy^{s}.
\end{equation*}
From the above expression, we see that
\begin{align*}
    (g_N)_{ij}=\delta_{ij}-\frac{1}{3}&W_{kijl}\frac{y^ky^l}{\abs{y}^4}-\frac{2}{\abs{y}^2}\big(\delta_{\alpha j}y^\alpha y^i+\delta_{i\beta}y^\beta y^j\big)+\frac{4}{\abs{y}^4}\delta_{\alpha\beta}y^\alpha y^\beta y^i y^j+\frac{2}{3}W_{k\alpha j l}\frac{y^k y^l}{\abs{y}^6}(y^\alpha y^i) \\
    &+\frac{2}{3}W_{k i \beta l }\frac{y^k y^l}{\abs{y}^6}(y^\beta y^j)-\frac{4}{3}W_{k\alpha\beta l}\frac{y^k y^l}{\abs{y}^8}(y^\alpha y^\beta y^i y^j)+O^{(3)}\big(\abs{y}^{-3}\big).
\end{align*}
However, the third term in the above expression simplifies with the fourth one and the last three terms are all zero due to the symmetries of curvature type tensors.
\end{proof}

\medskip

\begin{proof}[Proof of Lemma \ref{lem:wdot-TT}]
The proof is straightforward but tedious: all we have to do is rewrite each $\dot{w}_{ij}$ as a linear combination of terms {\em similar} to those of $H_{ij}$ or $F_{ij}$ defined in \eqref{eq:metr-err-def}.

\underline{Case $i=j$:}
    by \eqref{eq:w-wddot-relation}, the expression \eqref{eq:sol-ii-form} for $w_{ii}$ and formulae \eqref{eq:rad-funct-eq}, \eqref{eq:raf-fun-tilde} for the radial functions, we see that $\dot{w}_{ii}$ is equal to
    \begin{align*}
        \dot{w}_{ii}= a^{-2}\sum_{\sigma=1}^4\bigg(\frac{c_{\sigma,s,t}}{\abs{x}^{e(\sigma)}}\phi_{s,t}+\frac{c_{\sigma,s,u}}{\abs{x}^{e(\sigma)}}\phi_{s,u}+\frac{c_{\sigma,t,u}}{\abs{x}^{e(\sigma)}}\phi_{t,u}+\frac{\tilde{c}_{\sigma,s,u}}{\abs{x}^{e(\sigma)}}\psi_{s,u}+\frac{\tilde{c}_{\sigma,t,u}}{\abs{x}^{e(\sigma)}}\psi_{t,u}\bigg),
    \end{align*}
    where $e(1)=4, e(2)=2, e(3)=-2$, $e(4)=-4$ and $\{s,t,u\}$ are the other three distinct indices in $\{1,2,3,4\}$ which are different from $i$. Moreover, by \eqref{eq:syst-sol} we know that $\forall \sigma,\alpha,\beta$ one has
    \begin{align*}
        a^{-2}c_{\sigma,\alpha,\beta}&=C_{1,\sigma,\gamma}v_{1,\alpha,\beta} + C_{2,\sigma,\gamma} v_{2,\alpha,\beta}+C_{3,\sigma,\gamma}v_{3,\alpha,\beta} + C_{4,\sigma,\gamma} v_{4,\alpha,\beta} \\
        a^{-2}\tilde{c}_{\sigma,\alpha,\beta}&=C_{1,\sigma,\gamma}\tilde{v}_{1,\alpha,\beta} + C_{2,\sigma,\gamma} \tilde{v}_{2,\alpha,\beta}+C_{3,\sigma,\gamma}\tilde{v}_{3,\alpha,\beta} + C_{4,\sigma,\gamma} \tilde{v}_{4,\alpha,\beta},
    \end{align*}
    where the $C_{d,\sigma,\gamma}$ are constants which \emph{do not} depend upon the indices $\alpha,\beta$ (and do not depend upon the tilde). Hence $\dot{w}_{ii}$ is a linear combination of terms $T_{ii}=T_{ii}(\sigma,d,\gamma)$ of the following type:
    \begin{equation*}
        T_{ii}=\frac{v_{d,s,t}}{\abs{x}^{m}}\phi_{s,t}+\frac{v_{d,s,u}}{\abs{x}^{m}}\phi_{s,u}+\frac{v_{d,t,u}}{\abs{x}^{m}}\phi_{t,u}+\frac{\tilde{v}_{d,s,u}}{\abs{x}^{m}}\psi_{s,u}+\frac{\tilde{v}_{d,t,u}}{\abs{x}^{m}}\psi_{t,u},
    \end{equation*}
    where $d=1,\dots,4$ and $m=m(\sigma)\in\{4,2,-2,-4\}$. Taking the values of $v_{d,\alpha,\beta}$ from \eqref{eq:vec-ii-phi}, \eqref{eq:vec-ii-psi} and writing up the spherical harmonics in polynomial form, one finds the following expression:
    \begin{equation*}
        T_{ii}=\overline{C}\bigg(2W_{siit}\frac{x_sx_t}{\abs{x}^{m+2}}+2W_{siiu}\frac{x_sx_u}{\abs{x}^{m+2}}+2W_{tiiu}\frac{x_tx_u}{\abs{x}^{m+2}}+W_{siis}\frac{x_s^2-x_u^2}{\abs{x}^{m+2}}+W_{tiit}\frac{x_t^2-x_u^2}{\abs{x}^{m+2}}\bigg),
    \end{equation*}
    where $\overline{C}=\overline{C}(\sigma,d,\gamma,\lambda^2)$ is a multiplicative constant (notice that it does \emph{not} depend on $i$). Using the trace-free property of Weyl tensor we can write $T_{ii}$ as
    \begin{equation}\label{eq:term-ii}
        T_{ii}=\overline{C}W_{kiil}\frac{x^k x^l}{\abs{x}^{m+2}}.
    \end{equation}

    \underline{Case $i\not=j$:} similarly to the previous case, by \eqref{eq:w-wddot-relation}, the expression \eqref{eq:sol-ij-form} for $w_{ij}$ and \eqref{eq:rad-funct-eq}, \eqref{eq:raf-fun-tilde}, \eqref{eq:syst-sol}, one sees that $\dot{w}_{ij}$ can be written as a linear combination of terms $T_{ij}=T_{ij}(\sigma,d,\gamma)$ of type
    \begin{equation*}
        T_{ij}=\frac{v_{d,j,i}}{\abs{x}^m}\phi_{j,i}+\frac{v_{d,j,s}}{\abs{x}^m}\phi_{j,s}+\frac{v_{d,j,t}}{\abs{x}^m}\phi_{j,t}+\frac{v_{d,s,i}}{\abs{x}^m}\phi_{s,i}+\frac{v_{d,t,i}}{\abs{x}^m}\phi_{t,i}+\frac{v_{d,s,t}}{\abs{x}^m}\phi_{s,t}+\frac{\tilde{v}_{d,s,t}}{\abs{x}^m}\psi_{s,t},
    \end{equation*}
    where $s,t$ denote the remaining two indices in $\{1,2,3,4\}$ distinct from $i,j$, and, again, $d=1,\dots,4$ and $m=m(\sigma)\in\{4,2,-2,-4\}$. Taking the values of $v_{s,\alpha,\beta}$ from \eqref{eq:vec-ij-phi}, \eqref{eq:vec-ij-psi-hpi-st} and using the trace-free property of $W$ one finds that
    \begin{equation}\label{eq:term-ij}
        T_{ij}=\overline{C}W_{kijl}\frac{x^k x^l}{\abs{x}^{m+2}},
    \end{equation}
    where the $\overline{C}$'s are the same constants of above.

    By \eqref{eq:term-ii} and \eqref{eq:term-ij}, we see that $\dot{w}$ can be written as a linear combination of tensors of type $T(m)$, where 
    \begin{equation*}
        T(m)_{ij}:=W_{kijl}\frac{x^k x^l}{\abs{x}^{m+2}}, \qquad m\in\{4,2,-2,-4\}.
    \end{equation*}
    As a consequence, if we show that $T(m)$ is TT for any choice of $m$, then $\dot{w}$ will also be TT by linearity of the trace and divergence. The trace-free property of $T(m)$ follows from that of the Weyl tensor. For the divergence, one has
    \begin{align*}
        \delta(T(m))_j&=\sum_i\partial_i\Big(W_{kijl}\frac{x^k x^l}{\abs{x}^{m+2}}\Big)=\sum_i W_{kijl}\frac{\delta_i^k x^l+x^k \delta_i^l}{\abs{x}^{m+2}}-(m+2)\sum_iW_{kijl}\frac{x^k x^l x_i}{\abs{x}^{m+4}} \\
        &=\frac{1}{\abs{x}^{m+2}}\sum_i (W_{iijl}x^l+W_{kiji}x^k)-\frac{m+2}{\abs{x}^{m+4}}W_{kijl}x^k x^i x^l=0. 
    \end{align*}
    This concludes the proof.
\end{proof}

\printbibliography[heading=bibintoc]

\end{document}